\documentclass{amsart}
\usepackage{ amssymb }
\usepackage{xcolor}
\usepackage{graphicx}
\usepackage[T1]{fontenc}
\usepackage{ stmaryrd }
\usepackage{enumitem} 

\usepackage{tikz}
\usepackage[normalem]{ulem}
\usetikzlibrary{decorations.pathreplacing}
\usetikzlibrary{snakes}
\usepackage{pict2e}

\usepackage{hyperref}
\hypersetup{colorlinks=true,linkcolor=blue,citecolor=blue}

\newcommand{\IE}{{\mathbb E}}
\newcommand{\DP}{{\mathrm P}}
\newcommand{\DE}{{\mathrm E}}

\renewcommand{\b}{\beta}

\newcommand{\eqlaw}{\stackrel{(d)}{=}}

\newcommand{\dd}{\mathrm{d}}

\newcommand{\mR}{\mathcal R}

\newtheorem{theorem}{Theorem}[section]
\newtheorem{lemma}[theorem]{Lemma}
\newtheorem{proposition}[theorem]{Proposition}

\newtheorem{corollary}[theorem]{Corollary}

\newtheorem*{theorem*}{Theorem}
\newtheorem*{lemme*}{Lemme}

\newtheorem*{proposition*}{Proposition}

\newtheorem{remark}[theorem]{Remark}

\theoremstyle{definition}

\theoremstyle{remark}


\begin{document}
\title[Moments of polymers]{Moments of partition functions of 2D Gaussian polymers in the weak disorder regime - II}
\author{Cl\'ement Cosco and Ofer Zeitouni}
\address{Cl\'{e}ment Cosco,
Ceremade, Universite Paris Dauphine, Place du Mar\'{e}chal de Lattre de Tassigny, 75775 Paris Cedex 16, France}
\address{Ofer Zeitouni,
Department of Mathematics, Weizmann Institute of Sciences,
Rehovot 76100, Israel.}
\thanks{This project has received funding from the Israel Science Foundation
grant \#421/20.}
\email{clement.cosco@gmail.com, ofer.zeitouni@weizmann.ac.il}

\begin{abstract}
{Let $W_N(\b) = \DE_0\left[e^{ \sum_{n=1}^N \b\omega(n,S_n) - N\b^2/2}\right]$ be the partition function of a two-dimensional directed polymer in a random environment,
where $\omega(i,x), i\in \mathbb{N}, x\in \mathbb{Z}^2$ are i.i.d.\ standard normal and $\{S_n\}$ is the path of a random walk. With $\beta=\beta_N=\hat\b \sqrt{\pi/\log N}$ and $\hat \b\in (0,1)$ (the subcritical window),
$\log W_N(\b_N)$ is known to converge in distribution to a Gaussian  law of mean $-\lambda^2/2$ and variance $\lambda^2$, with $\lambda^2=\log (1/(1-\hat\b^2))$
 (\textit{Caravenna,  Sun, Zygouras, Ann.\ Appl.\ Probab.\ (2017)}).
We study in this paper
the moments $\IE [W_N( \b_N)^q]$ in the subcritical window, and prove a lower bound that matches for $q=O(\sqrt{\log N})$ the upper bound derived by 
us in \textit{Cosco, Zeitouni, arXiv:2112.03767}. 
The analysis is based on appropriate decouplings and a Poisson convergence
that uses the method of ``two moments suffice''.}
\end{abstract}

\maketitle

\section{Introduction and results}
Let $((S_n)_{n\geq 0},(\DP_x)_{x\in\mathbb Z^2})$ be the simple random walk on $\mathbb Z^2$. The associated expectation will be written as $\DE_x$. We let $p_n(x) = \DP_0(S_n=x)$.

Let $\omega(n,x)$, $n\in\mathbb N$, $x\in\mathbb Z^2$ be a collection of i.i.d.\ random variables distributed according to a centered Gaussian of variance one $\mathcal N(0,1)$.

{Set}
\begin{equation*} \beta_N = \frac{\hat \beta}{\sqrt{R_N}},\quad R_N = \DE_{0}^{\otimes 2} \left[ \sum_{n=1}^N \mathbf{1}_{S_n^1=S_n^2}\right]\sim \frac{1}{\pi}\log N.
\end{equation*}
where the asymptotics on $R_N$ follow from the local limit theorem $p_{2n}(0)\sim \frac 1 {\pi n}$, see e.g.\ Appendix \ref{AppendixLCLT}.
We define the \emph{normalized partition function}: 
\[W_N = \DE_0\left[ e^{\sum_{n=1}^N \beta_N \omega(n,S_n) - N \frac{\beta_N^2}{2}}\right].\]
 It is known, {see e.g. \cite[Theorem 2.8]{CaraSuZy-universalityrelev},
that for $\hat\beta<1$,} 
$\log W_N \to \mathcal N(-\lambda^2/2,\lambda^2)$,
where $\lambda^2 = \lambda(\hat \beta)^2 = -\log(1-\hat \beta^2)$, {and further, from
\cite[Theorem 1.1]{LyZy21},
we have that for any fixed $q$ integer and $\hat\beta<1$,}
\begin{equation} \label{eq-qmomentsfixed}
  \IE[W_N^q]\to_{N\to\infty} e^{\lambda^2\binom{q}{2}}.
\end{equation}

The goal of this paper is to establish a lower bound on the $q$-th moment of $W_N$ when $q$ can {increase as function of $N$, thus complementing the upper bounds
derived in \cite{CZ21}, to which we refer for motivation and applications. Of particular interest
is the case of $q^2$ of  order $\log N$.} Our starting point is the formula
\begin{equation} \label{eq:formula_moments}
  \IE[W_N^q]=\DE_0^{\otimes q} \left[e^{\beta_N^2 \sum_{(i,j)\in \mathcal C_q} \sum_{n=1}^N \mathbf{1}_{S^i_n = S^j_n}} \right],
\end{equation}
where
$\mathcal C_q = \{(i,j), 1\leq i<j\leq q\}$.
(See \cite{CZ21} for a proof of \eqref{eq:formula_moments}.) Here is our main result.
\begin{theorem} \label{th:main} Suppose that $q^2 =O(\log N)$ and $\log \log N = O(q^2)$. There exists $\varepsilon_N=\varepsilon_N(\hat \beta)\to 0$ as $N\to\infty$ such that
  $ \IE[W_N^q] \geq e^{\lambda^2 \binom{q}{2}(1-|\varepsilon_N|)}$.
  \end{theorem}
  This last bound {matches to leading order}
  the upper bound $\IE[W_N^q]\leq e^{\binom q 2 \lambda^2(1+|\varepsilon_N|)}$ that we obtained in \cite{CZ21} in the regime $q^2 \leq c \log N$ with $c=c(\hat \beta)$.

  {Theorem \ref{th:main} requires $q$ to be larger than $\sqrt{\log\log N}$.
  The statement in fact continues to hold without that restriction: indeed, for $q=O(1)$ this is contained in \cite{LyZy21}, while we provide in Appendix
  \ref{appendixB} the modifications needed to extend the statement
  to the range $1\ll q^2<\log\log N$.}

In fact,  for $q$ independent of $N$, the convergence \eqref{eq-qmomentsfixed} yields an exact equivalence with errors $o(1)$ in the exponents.
 As shown  in \cite{LyZy22}, the underlying reason 
 is an asymptotic decoupling for the intersection local time of the walks. In comparison, we prove a weaker form of decoupling, for a larger number of walks.

\begin{remark}
It was pointed to us by F. Caravenna that in the continuous setup, i.e. when the random walk $S_n$ is replaced by a Brownian motion,  the sum in the definition of $W_N$ is replaced by an integral, and the environment replaced by a regularized white noise, the result of Theorem \ref{th:main} with $\varepsilon_N=0$ 
 follows from a correlation inequality, see \cite{CSZNot} for a similar argument. We do not see how to adapt this to the discrete setup.
\end{remark}

We further observe that when $q$ is too large, the behavior changes:
\begin{theorem} \label{th:qlarge} For all $\hat \beta > 0$
  there exist $c_0=c_0(\hat\beta)>0$ and $c_1=c_1(\hat \beta)>0$ such that when $q^2 \geq c_1 (\log N)^2$, we have $\IE[W_N^q] \geq e^{c_0 \binom q 2 N /\log N}$. 
\end{theorem}

\subsection{A high level view of the proof and structure of the paper}
{We provide in this section a somewhat impressionistic view of the proof, that neglects
important details but captures the main ideas. The starting point is 
\eqref{eq:formula_moments}, that reduces the computation of moments of the partition function to the evaluation of exponential moments of the total
pairwise intersections of $q$ independent random walk paths. Toward this end,
we introduce 
certain decoupling time $L_k$ with $L_{k+1}=L_k+o_k$ and with
$o_k$ being a large multiple ($\nu_2$) of $l_k\gg1$, see \eqref{eq:Lk}.
Very roughly, $l_k\sim (cl_{k-1})^{1+\alpha/\log N}$, and we mostly care about
$l_k>N^\epsilon$ for some $\epsilon$ small. 
Now, within each interval $I_k=[L_k,L_{k+1})$, we only count intersections of 
paths within a subinterval of length $l_k$ that is separated from both ends, and within this interval we only count the intersections of
\textit{disjoint} pairs. Using the Markov property, contributions from
different $I_k$s decouple as long as we condition on the position 
of the paths at the beginning and end of $I_k$
(the precise statement is contained
in Proposition \ref{prop:lowerBoundDN_Upsilon}). Crucially, we then reduce 
the contribution within each $I_k$ to paths whose starting points and ending
points are ``where they should be'',  and then further reduce it to a moment
of a certain quantity we call $a_k$, see \eqref{def:a_k}, which depends only on a pair of random walks, and the total number of disjoint pairs that intersect, 
denoted $\mR_k$; this is the content of the crucial Proposition \ref{prop:iteration}.}

Having obtained the decoupling, there are two tasks remaining. The first is to obtain a good control on $a_k$, that is the contribution of intersections
of a single pair of walks. This 
necessitates estimates that are related to 
those we obtained in \cite{CZ21}, with the upshot being that that $a_k\sim 
{1}/({1-\hat \beta^2 (\log l_k)/(\log N)})$, see Proposition \ref{prop:ak}.

The main innovation of the paper is then to obtain a good control of $\mR_k$, the number of disjoint pair intersections. We prove in Proposition \ref{prop:2moments} that 
$\mR_k$ is close in distribution to a Poisson random variable. The proof of Proposition \ref{prop:2moments}, which takes up most of Section \ref{sec:estimate2moments}, is based on Stein's method,
more specifically on the ``two moments suffice'' theorem of Arratia, Goldstein and Gordon \cite{AGG89}. Essentially, we  use that disjoint pairs of path are independent to introduce a notion of neighborhood of dependence between
pairs of indices. Taking parameters in the right order drives the Poisson parameter (roughly, $\alpha$)  to infinity and completes the proof of Theorem 
\ref{th:main}. 

Theorem \ref{th:qlarge} 
is much easier and obtained by forcing an event where
the walks stay confined to a neighborhood of the origin. See Section \ref{sec:qlarge} for the proof.
\subsection{Notation} 
\label{subsec-notation}
Throughout the paper, constants $C$
are positive universal constants, 
whose values may change at different occurrences.

{We use various parameters, and limits in a particular order, that we now introduce. We use the parameters
$\gamma,\varepsilon_0,\delta \in (0,1)$ and 
$\alpha,\nu_1,\nu_2,M\in \mathbb N$, 
and the following order of successive limits:
 (i) $N\to\infty$, (ii) $\alpha \to\infty$, (iii) $\nu_1\to\infty$, (iv) $\nu_2\to\infty$, (v) $\delta\to 0$, (vi) $M\to\infty$, (vii) $\varepsilon_0\to 0,\gamma\to 0$.  (The last limit can be taken simultaneously for $\varepsilon_0$ and $\gamma$). We introduce the collection of variables
 $\widehat{\Gamma}=(M,\delta,\nu_2,\nu_1,\alpha)$, 
$\widetilde{\Gamma}=(\gamma,\varepsilon_0,M,\delta,\nu_2,\nu_1,\alpha)$
and $\Gamma'=(N,\widehat{\Gamma})$, $\Gamma=(N,\widetilde{\Gamma})$.
 For any function $\Psi$, we let $\limsup_{\Gamma} \Psi(\Gamma)$
denote the limsup obtained after taking successive limsups 
in the order described above. 
We define $\limsup_{\Gamma'} \Psi(\cdot)$, 
$\limsup_{\widetilde{\Gamma}} \Psi(\cdot)$ and 
$\limsup_{\widehat{\Gamma}} \Psi(\cdot)$ similarly.}

We will use repeatedly that $(S_k^1-S_k^2) \eqlaw (S_{2k})$ when $S^1_n$ and $S^2_n$ are two independent simple random walks.

$B(x,r)$ denotes the Euclidean ball
of radius $r$ centered at $x\in \mathbb R^2$.

\section{Proofs}
\subsection{Preliminaries for the proof of Theorem \ref{th:main}}

Throughout the paper, we always assume that $N,\varepsilon_0^{-1},\delta^{-1},\nu_1,\nu_2,M,\alpha \geq 100$ and in accordance to the order of the limits, that
\begin{equation} \label{eq:assumptionOnParameters}
\begin{aligned}
 & \  (i)\ \delta^{-2} e^{-\frac{1}{2} \gamma \alpha} \leq 1,\quad (ii)\ \frac{\log (4\nu_1)}{\alpha \gamma} < 2^{-2}, \quad (iii) \ \nu_1^{-1} \delta^{-2} \nu_2 \leq 2^{-5},\\
&  (iv)\ \nu_2 e^{-\gamma \alpha/2} M^{-2} \leq 2^{-4}, \quad (v)\ N>(4\nu_2)\vee e^{2\alpha}.
\end{aligned}
\end{equation}

Next, we introduce the times $l_k,L_k$ that we use to decompose the process. With
\begin{equation}
\bar \alpha = \alpha/\log N \quad \text{and} \quad f_k = e^{k\bar \alpha},
\end{equation}
 we set: 
 \begin{equation}
   \label{eq:Lk}
   l_k= \left\lceil N^{\gamma f_k}\right\rceil, \quad \quad o_k =  \nu_1 l_{k-1} + 2l_k + \nu_2 l_k \quad  \text{and} \quad L_k=\sum_{1\leq j< k} o_j,\end{equation}
for all $k\in \llbracket 1, K\rrbracket$, where $K = \max\{k\in \mathbb Z_+, L_{k+1} \leq N\}$. (We set $l_0=L_1 = 0$).

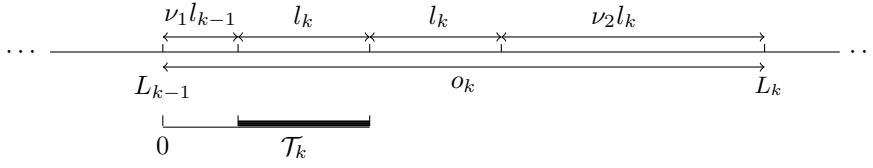
\begin{figure}[h]
\begin{center}
\begin{tikzpicture}

  \begin{scope}
    \draw [-]   (-6.5,0) -- (4.0,0) node[right] {$\cdots$ } ;
    \draw[<->] (-5,-0.2)--(3,-0.2);
    \coordinate (AAA) at (-1,-0.2);
    \draw (AAA) node[below] {$o_{k}$};
\foreach \x in {-5,-4,-2.25,-0.5,3}
\draw (\x,3pt)--(\x,0pt);
    \draw [-]   (-5,-1) -- (-2.25,-1);

    \draw [-,thick] (-4,-0.95)--(-2.25,-0.95);
    \draw [-,thick] (-4,-0.975)--(-2.25,-0.975);
    \draw [-,thick] (-4,-0.925)--(-2.25,-0.925);
\foreach \x in {-5,-4,-2.25}
\draw (\x,-28pt)--(\x,-24pt);
\coordinate (W) at (-5,-1);
\draw (W) node[below]{$0$};
\coordinate (WW) at (-3.25,-1);
\draw (WW) node[below]{$\mathcal{T}_k$};
\coordinate (A) at (-5,-0.2);
\draw (A) node[below] {$L_{k-1}$};
\coordinate (AA) at (3,-0.2);
\draw (AA) node[below] {\small{ $L_{k}$}};
\coordinate (B) at (-6.5,0);
\draw (B) node[left] {$\cdots$};

    \draw[<->] (-5,0.2)--(-4,0.2);
    \coordinate (BBB) at (-4.5,0.2);
    \draw (BBB) node[above] {$\nu_1 l_{k-1}$};
    \draw[<->] (-4,0.2)--(-2.25,0.2);
    \coordinate (BB) at (-3.125,0.2);
    \draw (BB) node[above] {$l_{k}$};
    \draw[<->] (-2.25,0.2)--(-0.5,0.2);
    \coordinate (BB) at (-1.325,0.2);
    \draw (BB) node[above] {$l_{k}$};
    \draw[<->] (-0.5,0.2)--(3,0.2);
    \coordinate (CC) at (1,0.2);
    \draw (CC) node[above] {$\nu_2 l_{k}$};

  \end{scope}

\end{tikzpicture}
\caption{Pictorial description of $k$th intervals}
\label{fig:casesshort}
\end{center}
\end{figure}

The times $L_k$ and $l_k$ satisfy the following straightforward relations:
\begin{lemma} \label{lem:estimates}
For all $k\leq K$:
\begin{equation} \label{eq:boundl_k}  (i)\ e^{\gamma \alpha/2} \leq \frac{l_{k+1}}{l_k} \leq e^{e \alpha }, \qquad (ii)\ L_{k+1} \leq 4 \nu_2 l_k.
\end{equation}
Moreover, the following bounds on $K$ hold:
\begin{equation} \label{eq:K}
  \bar \alpha^{-1}\left(\log \gamma^{-1}+\log\left(1-\frac{\log (4\nu_2)}{\log N}\right)\right)\leq K\leq \bar \alpha^{-1}\log\gamma^{-1}.
\end{equation}
\end{lemma}
\begin{remark} \label{rem:2lk}
  It follows from \eqref{eq:boundl_k}-(i) and  \eqref{eq:assumptionOnParameters}-(ii) that $\nu_1 l_{k-1} \leq l_k$. This fact will turn out useful in several places. 
\end{remark}
\begin{proof}
We first show \eqref{eq:boundl_k}.
By rounding effects
$N^{\gamma f_k}\leq l_k \leq N^{\gamma f_k}(1+N^{-\gamma})$,
hence using that $\frac{N^{\gamma f_{k+1}}}{N^{\gamma f_k}} = N^{\gamma f_{k}(e^{\bar \alpha}-1)}$, it follows that
\begin{equation} \label{eq:ratioOnl_k}
(1+N^{-\gamma})^{-1-(e^{\bar \alpha}-1)}l_k^{(e^{\bar \alpha}-1)}\leq \frac{l_{k+1}}{l_k}\leq l_k^{(e^{\bar \alpha}-1)}(1+N^{-\gamma}).
\end{equation}
As by definition $N^{\gamma} \leq l_k \leq N$, the usual estimate $ \bar \alpha \leq e^{\bar \alpha} - 1\leq \bar \alpha e^{\bar \alpha}$ and \eqref{eq:assumptionOnParameters}-(v) yield that $e^{\gamma \alpha} \leq l_k^{(e^{\bar \alpha}-1)} \leq e^{e^{1/2}\alpha}$. We then bound $(1+N^{-\gamma})$ by 2 and obtain \eqref{eq:boundl_k}-(i) from \eqref{eq:ratioOnl_k} by using that $\alpha$ and $\gamma \alpha$ are large by \eqref{eq:assumptionOnParameters}-(i).

  
  Now, equation \eqref{eq:boundl_k}-(i) implies that for all $j\leq k$, we have
  $l_j\leq e^{-\gamma \alpha  (k-j)/2} l_k$. Therefore,
  \begin{align*}
  L_{k+1} & =(\nu_1 + 2 + \nu_2) \sum_{1\leq j<k} l_j + (2+\nu_2) l_k\\
  & \leq (\nu_1 + 2 + \nu_2) \sum_{1\leq j<k} e^{-\gamma \alpha  (k-j)/2}  l_k + (2+\nu_2) l_k\leq (\eta + 2 + \nu_2)l_k,
  \end{align*}
  with $\eta = \frac{e^{-\gamma \alpha/2}}{1-e^{-\gamma \alpha /2}}(\nu_1+ 2+\nu_2)$. We find \eqref{eq:boundl_k}-(ii) via \eqref{eq:assumptionOnParameters}-(ii).

Regarding \eqref{eq:K}, the upper bound on $K$ is obtained using that $L_k\geq l_k$. The lower bound is a consequence of \eqref{eq:boundl_k}-(ii) and \eqref{eq:assumptionOnParameters}-(v).
\end{proof}

To help us control the positions of the walks at the times $(L_k)$, we define the (random) set of indices
\[G_k = \left\{i\in\llbracket 1,q\rrbracket \, :\, S_{L_k}^i \in B\left(0,\delta^{-1} L_k^{1/2}\right) \text{ and } S_{L_{k+1}}^i \in B\left(0,\delta^{-1} L_{k+1}^{1/2}\right)\right\},\]
 where {we recall that
 $B(x,r)$ is the Euclidean 
 ball of radius $r$ centered at $x\in \mathbb R^2$,}
 and further introduce the event:
\[A_k = \left\{|G_k| \geq (1-\varepsilon_0)q \right\}.\] For all $m\in \mathbb N$ and $\mathbf x=(x_1,\dots,x_m),\mathbf y=(y_1,\dots,y_v)\in (\mathbb Z^2)^m$, write $\mathbf x\sim_{n} \mathbf y$ whenever $\DP_{\mathbf x}^{\otimes m}(S^1_n=y_1,\dots,S^m_n=y_m)> 0$. When $\mathbf x \sim_n \mathbf y$, denote by $\DE_{\mathbf x}^{n,\mathbf y}$ the expectation for $m$ copies of the simple random walk started at $\mathbf x$ and conditioned on arriving at $\mathbf y$ at time $n$, that is
\[\DE_{\mathbf x}^{n,\mathbf y}[\cdot] = \DE_{\mathbf x}^{\otimes m} [\cdot | S_n^1 = y_1,\dots,S_n^m = y_m].\] 
Further let 
$B_{m,k}= \left(B\left(0,\delta^{-1} L_{k}^{1/2}\right)\cap  \mathbb Z^2\right)^m$.

We are now ready to decompose the moment of $W_N$ as a product of contributions coming from the different time intervals $[L_k,L_{k+1}]$. This is the purpose of the next proposition.
\begin{proposition} \label{prop:lowerBoundDN_Upsilon}
  Let ${q_0}=\lfloor (1-\varepsilon_0) q \rfloor$. We have:
\begin{equation} \label{eq:first_lowerbound}
  \IE\left[ W_N^q\right]  \geq  D_N \prod_{k=1}^K \Upsilon_k,
\end{equation}
where $D_N:=\DE_0^{\otimes q}\left[\prod_{k=1}^K \mathbf{1}_{A_k} \right]$ and 
\begin{equation} \label{eq:defUpsilon}
  \Upsilon_k := \inf_{\substack{\mathbf x\in B_{{q_0},k},\mathbf y\in B_{{q_0},k+1}\\ \mathbf x \sim_{o_k} \mathbf y}} \DE_{\mathbf x}^{o_k,\mathbf y}\left[e^{\beta_N^2  \sum_{n=\nu_1 l_{k-1}}^{\nu_1 l_{k-1} + 2l_k} \sum_{(i,j)\in \mathcal C_{{q_0}}} \mathbf{1}_{S^i_n = S^j_n}}\right].
\end{equation}
\end{proposition}
\begin{proof} 
  Let $\Psi_{L} = e^{\beta_N^2 \sum_{(i,j)\in \mathcal C_q} \sum_{n=1}^{L} \mathbf{1}_{S^i_n = S^j_n}}$.
We will prove by induction that for all $l\in \llbracket 0,K\rrbracket$,
\begin{equation*} \label{eq:inductionHyp}
\mathcal H_l:\quad \IE\left[ W_N^q\right]  \geq \DE_0^{\otimes q} \left[\Psi_{L_{K+1-l}} \prod_{k=1}^{K} \mathbf{1}_{A_k} \right] \prod_{k=K+1-l}^{K} \Upsilon_k.
\end{equation*}
The case $l=K$ will then give the proposition (recall that $L_1=0$).

First, $\mathcal H_0$ holds by \eqref{eq:formula_moments} (we use the convention that an empty product equals 1). Suppose now that $\mathcal H_l$ holds for some $l<K$. Let $\mathbf S_n = (S^1_n,\dots,S^q_n)$ and denote by $\tilde A_k$ the event $A_k$ shifted in time by $-L_{K-l}$. By Markov's property,
\begin{equation} \label{eq:FirstMarkov}
\DE_0^{\otimes q} \left[\Psi_{L_{K+1-l}} \prod_{k=1}^{K} \mathbf{1}_{A_k} \right]= \DE_0^{\otimes q} \left[\Psi_{L_{K-l}}\hspace{-2mm} \prod_{k=1}^{K-l-1} \hspace{-2mm} \mathbf{1}_{A_k} \DE^{\otimes q}_{\mathbf S_{L_{K-l}}} \left[ \prod_{k=K-l}^K \mathbf{1}_{\tilde A_k} \Psi_{o_{K-l}} \right]\right].
\end{equation}
(recall that $o_{K-l} = L_{K+1-l} - L_{K-l}$).
We apply again Markov's property to find that for all $\mathbf x = (x_1,\dots,x_q) \in (\mathbb Z^2)^q$,
\begin{align*}
&\DE^{\otimes q}_{\mathbf x} \left[ \prod_{k=K-l}^K \mathbf{1}_{\tilde A_k} \Psi_{o_{K-l}} \right]  = \DE^{\otimes q}_{\mathbf x} \left[\mathbf{1}_{\tilde A_{K-l}} \DE^{\otimes q}_{\mathbf x} \left[\Psi_{o_{K-l}}\middle| \mathbf S_{o_{K-l}} \right]  \DE^{\otimes q}_{\mathbf S_{o_{K-l}}} \left[ \prod_{k=K-l+1}^K \mathbf{1}_{\tilde A_k} \right] \right].
\end{align*}
On the event $\tilde A_{K-l}$, we let 
$(i_r)_{r\leq {q_0}}$ be {the ${q_0}$ smallest}
indices such that $S_0^{i_r} \in B(0,\delta^{-1} L_{K-l}^{1/2})$ and  $S_{o_{K-l}}^{i_r} \in B(0,\delta^{-1} L_{K-l+1}^{1/2})$ for all $r\leq {q_0}$.
 It follows that on $\tilde A_{K-l}$, one has $\DE^{\otimes q}_{\mathbf x} \left[\Psi_{o_{K-l}}\middle| \mathbf S_{o_{K-l}} \right] \geq \Upsilon_{K-l}$ by restricting the sum inside the exponential to the walks indexed by the $i_r$'s and to the time interval $\llbracket \nu_1 l_{k-1},\nu_1 l_{k-1} + 2l_k\rrbracket$. In particular, we obtain from the last display that
\[\DE^{\otimes q}_{\mathbf S_{L_{K-l}}} \left[ \prod_{k=K-l}^K \mathbf{1}_{\tilde A_k} \Psi_{o_{K-l}} \right] \geq \Upsilon_{K-l} \DE^{\otimes q}_{\mathbf S_{L_{K-l}}} \left[ \prod_{k=K-l}^K \mathbf{1}_{\tilde A_k} \right].
\]
This combined with  $\mathcal H_{l}$ and \eqref{eq:FirstMarkov} implies that $\mathcal H_{l+1}$ holds.
\end{proof}
The goal now is to obtain a good lower bound on the quantity $\Upsilon_k$ defined in \eqref{eq:defUpsilon}. For this purpose, we introduce $\mathcal T_k$ the time interval:
\begin{equation} \label{eq:DefTk}
  \mathcal T_k = \llbracket \nu_1 l_{k-1},\nu_1  l_{k-1} + l_k \rrbracket,
\end{equation}
and define $\mR_k$ as the maximal number of disjoint pairs $(i,j)\in 
\mathcal C_{{q_0}}$ such that $S^i$ and $S^j$ intersect during $\mathcal T_k$ without leaving some large ball. More precisely, let 
\begin{equation}
  \sigma^i_k = \inf \left\{n\in \mathcal T_k, |S_n^i| > M l_k^{1/2} \right\},
\end{equation}
(we set $\sigma^i_k=\infty$ when the set is empty)
and define
\[\tau_1 = \inf\left\{n\in {\mathcal T_k} : \exists 
(i,j)\in\mathcal C_{{q_0}} 
\text{ such that } S^i_n = S^j_n \text{ and } n<  \sigma^i_k \wedge \sigma^j_k\right\},\]
 as the first time two particles intersect before one of them leaves the ball of radius $M l_k^{1/2}$ and $(I_1,J_1)$ be the two particles involved. If the set is empty, we let $\tau_1 = \infty$. Then, define iteratively: 
 \begin{align*}\tau_{r+1}  = \inf\big\{& n> \tau_{r},\, n\in \mathcal T_k : \exists (i,j)\in\mathcal C_{{q_0}} \text{ such that: }S^i_n = S^j_n, \\
  & \quad n<  \sigma^i_k \wedge \sigma^j_k \text { and }  \forall s\leq r, \{i,j\}\cap \{I_s,J_s\} = \emptyset
\Big\},
\end{align*}
as the next time two new particles, distinct from all the previous particles $I_1,J_1,$ $ \dots,$ $I_r,J_r$,  meet.
We denote by $(I_{r+1},J_{r+1})$ this new pair. When there is no such time, we set $\tau_{r+1} = \infty$.  Finally, denote by
\[\mR_k = \sup \{r\geq 0 : \tau_r < \infty\},\]
the total number of successive disjoint intersections.
{Note that the $\tau_r$ depend on $k$, however we supress this dependence in the notation.}
\begin{remark} \label{rk:maximalityOfRk}
  Note that {a consequence of the definition is that
  $\mR_k$ is maximal in the sense that any sequence of 
  disjoint intersecting couples 
  $\mathcal I'= ((I_1',J_1'),\dots,(I_r',J_r'))$ satisfies $r\leq \mR_k$.} 
\end{remark}
{Introduce the expression}
\begin{equation} \label{def:a_k}
\begin{aligned}
  &a_k := \\
  &\inf_{t\in \mathcal T_k} \inf_{\substack{x\in B(0,M l_k^{1/2})\\y_1,y_2\in B_{1,k+1},y_1,y_2\sim_{(o_k-t)} x}} \DE_x^{\otimes 2}\left[e^{\beta_N^2\sum_{n=1}^{l_k} \mathbf{1}_{S_n^1=S_n^2} }\middle|S^1_{o_k-t} = y_1,S^2_{o_k-t} = y_2\right],
\end{aligned}
\end{equation}
{The quantity $a_k$ will serve below as a lower bound on the (multiplicative)
contribution
of a couple $(I_r,J_r)$ to the total expectation.
Considering that we have $\mR_k$ such contributions, we now prove
the following result.}
\begin{proposition} \label{prop:iteration}
  {With notation as above, we have that for}
  all $k\in \llbracket 1, K\rrbracket$,
\begin{equation*}
\begin{aligned}
\Upsilon_k
\geq \inf_{\substack{\mathbf x\in B_{{q_0},k},\mathbf y\in B_{{q_0},k+1}\\ \mathbf x \sim_{o_k} \mathbf y}} \DE_{\mathbf x}^{o_k,\mathbf y} \left[ a_k^{\mR_k} \right].
\end{aligned}
\end{equation*}
\end{proposition}
\begin{proof}
As $\tau_r\in \mathcal T_k$, we have $[\tau_r+1,\tau_r+l_k]\subset [\nu_1 l_{k-1},\nu_1 l_{k-1} + 2l_k]$ (see Figure \ref{fig:casesshort}) so that
\[\sum_{n=\nu_1 l_{k-1}}^{\nu_1 l_{k-1} + 2l_k} \sum_{(i,j)\in \mathcal C_{{q_0}}} \mathbf{1}_{S^i_n = S^j_n}\geq \sum_{r=1}^{\mR_{k}} \sum_{n=\tau_{r}+1}^{\tau_r +  l_k} \mathbf{1}_{S^{I_r}_n = S^{J_r}_n}. \]
Therefore, it holds that 
\begin{equation} \label{eq:lb_FLKconditioning}
  \DE_{\mathbf x}^{o_k,\mathbf y}\left[e^{\beta_N^2  \sum_{n=\nu_1 l_{k-1}}^{\nu_1 l_{k-1} + 2l_k} \sum_{(i,j)\in \mathcal C_{{q_0}}} \mathbf{1}_{S^i_n = S^j_n}}\right] \geq  \DE_{\mathbf x}^{o_k,\mathbf y} \left[ \prod_{r=1}^{\mR_k} f(\tau_r,\tau_r + l_k,I_r,J_r)\right],
\end{equation}
where $f(s,t,i,j) = \exp\big({\beta_N^2\sum_{n=s+1}^t \mathrm{1}_{S_n^i=S_n^j}}\big)$. 
Recall the definition of $a_k$ in \eqref{def:a_k}. Our goal is to show that for all $R\geq 0$,
\begin{equation} \label{eq:factorization_R}
  \Phi_R:= \DE_{\mathbf x}^{o_k,\mathbf y} \left[ \prod_{r=1}^{R}  f(\tau_r,\tau_r + l_k ,I_r,J_r) \mathbf{1}_{\mR_k=R} \right] \geq a_k^R \DP_{\mathbf x}^{o_k,\mathbf y}(\mR_k=R).
\end{equation} 
{(Again, $\Phi_R$ depends on $k, {\mathbf x}, {\mathbf y}$,
but we supress this from the notation.)}
The equation \eqref{eq:factorization_R} holds trivially for $R=0$. Now suppose $R\geq 1$.
 Let $\mathcal F_n$ denote the sigma-algebra generated by the walks until time $n$ and denote by $\mathcal F_{\tau_1}$ the sigma-field stopped by $\tau_1$. Observe that by independence of the random walks and Markov's property,
\begin{align*}
\Phi_R & = \DE_{\mathbf x}^{o_k,\mathbf y} \left[\mathbf{1}_{\tau_1<\infty}\DE_{\mathbf x}^{o_k,\mathbf y} \left[ \prod_{r=1}^{R}  f(\tau_r,\tau_r + l_k ,I_r,J_r) \mathbf{1}_{\mR_k=R} 
\middle| \mathcal F_{\tau_1} \right]\right]\\
& = \DE_{\mathbf x}^{o_k,\mathbf y} \Bigg[
\mathbf{1}_{\tau_1<\infty} \DE^{\otimes 2}_{S_{\tau_1}^{I_1},S_{\tau_1}^{J_1}}\left[f\left(0,l_k,1,2\right)\middle| S^1_{o_k-\tau_1} = y_{I_1},S^2_{o_k-\tau_1} = y_{J_1}\right]\times \\
&\hspace{1.7cm}   \DE^{o_k-\tau_1,(y_i)_{i\in \mathcal C_{{q_0}}\setminus \{I_1,J_1\}}}_{(S^i_{\tau_1})_{i\in \mathcal C_{{q_0}}\setminus \{I_1,J_1\}}} \left[ \prod_{r=1}^{R-1}  f(\tilde{\tau}_r,\tilde{\tau}_r + l_k,\tilde I_r,\tilde J_r)  \mathbf{1}_{\tilde \mR_k=R-1}  \right] \Bigg],
\end{align*}
where $\tilde \tau_r$, $\tilde I_r,\tilde J_r,\tilde \mR_k$ are defined as $\tau_r$, $I_r,J_r,\mR_k$ but for {$q_0-2$} particles and with $\mathcal T_k$ replaced by $\llbracket 0,\nu_1 l_{k-1}+l_k-\tau_1\rrbracket$.
As by definition $S_{\tau_1}^{I_1}=S_{\tau_1}^{J_1}\in B(0,M l_k^{1/2})$ and $\tau_1\in \mathcal T_k$, we obtain that
\begin{align*}
\Phi_R & \geq a_k \times \DE_{\mathbf x}^{o_k,\mathbf y} \Bigg[\mathbf{1}_{\tau_1<\infty}  \DE^{o_k-\tau_1,(y_i)_{i\in \mathcal C_{{q_0}}\setminus \{I_1,J_1\}}}_{(S^i_{\tau_1})_{i\in \mathcal C_{{q_0}}\setminus \{I_1,J_1\}}} \left[ \prod_{r=1}^{R-1}  f(\tilde{\tau}_r,\tilde{\tau}_r + l_k,\tilde I_r,\tilde J_r)  \mathbf{1}_{\tilde \mR_k=R-1}  \right]\Bigg]\\
& = a_k \DE_{\mathbf x}^{o_k,\mathbf y} \left[ \prod_{r=2}^{R}  f(\tau_r,\tau_r + l_k ,I_r,J_r) \mathbf{1}_{\mR_k=R} \right],
\end{align*}
where in the 
equality we have used Markov's property as above in the reverse direction. Iterating this process leads to \eqref{eq:factorization_R}. Then, putting together \eqref{eq:lb_FLKconditioning} and  \eqref{eq:factorization_R} and summing over $R$ entails Proposition \ref{prop:iteration}.
\end{proof}

Next, we define:
\begin{equation} \label{eq:def_lambdak2}
\lambda_k^2 = \log \frac{1}{1-\hat \beta^2 \frac{\log l_k}{\log N}}.
\end{equation}

\begin{proposition} \label{prop:ak} We have $\inf_{k\leq K} \{a_k - e^{\lambda_k^2}\} \geq -{\Delta_{\Gamma,{\ref{prop:ak}}}}$, where 
  $\Delta_{\Gamma,{\ref{prop:ak}}} >0$ satisfies {$\limsup_{\Gamma'}  |\Delta_{\Gamma,\ref{prop:ak}}| = 0$.}
\end{proposition}
\noindent
{(Recall that $\limsup_{\Gamma'}$ keeps $\gamma$ and $\varepsilon_0$ fixed
when taking the limsup, see Section \ref{subsec-notation}.)}
\begin{proof}
  {Throughout the proof, we write $\Delta_\Gamma$ instead of 
  $\Delta_{\Gamma,\ref{prop:ak}}$.}
Let $t\in \mathcal T_k$, $x\in B(0,M l_k^{1/2})$ and $y_1,y_2\in B_{1,k+1}$ such that $y_1,y_2\sim_{(o_k-t)} x$.
Let 
\[W(z_1,z_2) = \DE_x^{\otimes 2}\left[e^{\beta_N^2\sum_{n=1}^{l_k} \mathbf{1}_{S_n^1=S_n^2} }\mathbf{1}_{S^1_{l_k} = z_1} \mathbf{1}_{S^2_{l_k} = z_2}\right],\]
where we have supressed the dependence on $x$
and $k$ in the notation. 
By Markov's property,
 \begin{equation} \label{eq:Markovak}
\DE_x^{\otimes 2}\left[e^{\b_N^2\sum_{n=1}^{l_k} \mathbf{1}_{S_n^1=S_n^2} }\mathbf{1}_{S^1_{o_k-t} = y_1} \mathbf{1}_{S^2_{o_k-t} = y_2}\right]
= \sum_{z_1,z_2\in \mathbb Z^2} W(z_1,z_2) \prod_{i=1,2} p_{o_k-t-l_k}(y_i-z_i).
\end{equation}
We first show that when $|z_1|\wedge|z_2| \leq 2Ml_{k}^{1/2}$ and $z_i\sim_{o_k-t-l_k}y_i$,
\begin{equation} \label{eq:lowerBoundFromLLT}
\prod_{i=1,2} p_{o_k-t-l_k}(y_i-z_i) \geq e^{-\theta_\Gamma} \prod_{i=1,2} p_{o_k-t}(y_i-x),
\end{equation}
where, for some $\varepsilon_N = \varepsilon_N(\tilde \Gamma)$ that vanishes as $N\to\infty$,
\begin{equation} \label{eq:defAlpha'}
  \theta_\Gamma = |\varepsilon_N| + 2 b_1 - \log(1-6\nu_2^{-1}), \quad b_1 = 20\delta^{-1} \nu_2^{-1/2}M + 10M^2 \nu_2^{-1}.
\end{equation}  
To show \eqref{eq:lowerBoundFromLLT}, we rely on the local central limit theorem given in Appendix \ref{AppendixLCLT}. 
First observe that $o_k-t-l_k \geq \nu_2 l_k$ when $t\in \mathcal T_k$. We will use this repeatedly. Moreover, for $|z_1|{\vee}|z_2| \leq 2Ml_{k}^{1/2}$ and $t,x,y_i$ as above, we have that $|y_i-z_i|$ and $|y_i-x|$ are less than $(2\delta^{-1} \nu_2^{1/2} + 2M)l_{k}^{1/2}$ by \eqref{eq:boundl_k}-(ii). Since $l_k \geq N^{\gamma}$, we obtain that $|y_i-z_i| \leq c_N (o_k-t-l_k)$ and $|y_i-x| \leq c_N (o_k-t)$ with $c_N$ vanishing as $N\to\infty$. Hence Theorem \ref{th:LCLT} applies and we obtain that
\begin{align*}
&p_{o_k-t-l_k}(y_i-z_i) = 2 \bar{p}_{o_k-t-l_k}(y_i-z_i) e^{O(d_k)},\\
&p_{o_k-t}(y_i-x) = 2 \bar{p}_{o_k-t}(y_i-x) e^{O(d_k)},
\end{align*}
where $\bar p_{s}(z) = \frac{1}{\pi s} e^{-|z|^2/s}$ and $d_k =  \frac{1}{\nu_2 l_k} + \frac{\delta^{-4}\nu_2^2 + M^4}{\nu_2 ^3l_k}$.
Note that $d_k\leq c N^{-\gamma}$ with a constant $c$ depending on $\delta,\nu_2$ and $M$. 
Then, one finds by a simple computation that for $\mathbf x = (x,x)$,
\begin{align}
\label{eq-ratiop}
&\frac{\bar p_{o_k-t-l_k}(y_1-z_1) \bar p_{o_k-t-l_k}(y_2-z_2)}{\bar p_{o_k-t}(y_1-x) \bar p_{o_k-t}(y_2-x)} \nonumber\\
&= \frac{(o_k-t)^2}{(o_k-t-l_k)^2}e^{-(|y_1|^2+|y_2|^2)((o_k-t)^{-1}-(o_k-t-l_k)^{-1})+\frac{g(\mathbf z,\mathbf y)}{o_k-t-l_k} - \frac{g(\mathbf x,\mathbf y)}{o_k-t}}\\
&{\geq \frac{(o_k-t)^2}{(o_k-t-l_k)^2}e^{\frac{g(\mathbf z,\mathbf y)}{o_k-t-l_k} - \frac{g(\mathbf x,\mathbf y)}{o_k-t}},}\nonumber
\end{align}
where $g(\mathbf z,\mathbf y) = 2\langle y_1,z_1\rangle + 2\langle y_2,z_2\rangle - |z_1|^2 - |z_2|^2$ 
{and
we used in the last inequality that $(o_k-t-l_k)^{-1}>(o_k-t)^{-1}$.} 
Recall $b_1$ from \eqref{eq:defAlpha'}.
By the Cauchy-Schwarz inequality, the absolute value of {each of the}
 two terms in the last exponential is smaller than
\[(\nu_2 l_k)^{-1}\left(10\delta^{-1}L_{k+1}^{1/2} M l_k^{1/2} + 10M^2 l_k\right) \leq b_1.\]
Moreover, 
\[
  \left|\frac{(o_k-t)^2}{(o_k-t-l_k)^2} -1\right| = \frac{l_k(2(o_k-t)-l_k)}{(o_k-t-l_k)^2}  \leq \frac{l_k(6\nu_2 l_k)}{(\nu_2l_k)^2} \leq 6\nu_2^{-1}.
\]
Putting things together leads to \eqref{eq:lowerBoundFromLLT}.

Coming back to \eqref{eq:Markovak}, the bound \eqref{eq:lowerBoundFromLLT} entails that
\begin{equation}\label{eq:biggerthansumWz}
\DE_x^{\otimes 2}\left[e^{\beta_N^2 \sum_{n=1}^{l_k} \mathbf{1}_{S_n^1=S_n^2} }\middle|S^1_{o_k-t} = y_1,S^2_{o_k-t} = y_2\right] \geq e^{-\theta_\Gamma}\hspace{-5mm} \sum_{\substack{z_1,z_2\in \mathbb Z^2\\|z_1|\wedge|z_2|\leq M l_{k}^{1/2}}}W(z_1,z_2).
\end{equation}
We have
\begin{equation} \label{eq:sumWz}
  \sum_{\substack{z_1,z_2\in \mathbb Z^2\\|z_1|\wedge|z_2|\leq M l_{k}^{1/2}}}W(z_1,z_2) {\geq }
  \DE_{x}^{\otimes 2} \left[e^{\beta_N^2\sum_{n=1}^{l_k}\mathbf{1}_{S_n^1=S_n^2}}\right] - 2\sum_{\substack{z_1,z_2\in \mathbb Z^2\\
  |z_1| > M l_{k}^{1/2}}}W(z_1,z_2),
\end{equation}
where
\[\sum_{\substack{z_1,z_2\in \mathbb Z^2\\|z_1| >2Ml_{k}^{1/2}}} W(z_1,z_2) =  \DE_x^{\otimes 2}\left[e^{\beta_N^2\sum_{n=1}^{l_k} \mathbf{1}_{S_n^1=S_n^2} }\mathbf{1}_{\left|S^1_{l_k}\right| >2Ml_{k}^{1/2}} \right].\]
Recall the definition of $\lambda_k^2$ in \eqref{eq:def_lambdak2}. Given that $l_k\geq N^{\gamma}$, one can see from the proof of Proposition 3.4 in \cite{CZ21} that there exists $\varepsilon_N' =\varepsilon_N'(\gamma)\to 0$ as $N\to\infty$ such that
\[ \DE_{x}^{\otimes 2} \left[e^{\beta_N^2\sum_{n=1}^{l_k}\mathbf{1}_{S_n^1=S_n^2}}\right] =\DE_{0} \left[e^{\beta_N^2\sum_{n=1}^{l_k}\mathbf{1}_{S_{2n} = 0}}\right] \geq (1+\varepsilon_N') e^{\lambda_k^2}.\]
Moreover, by H\"older's inequality with $p^{-1}+(p')^{-1}=1$ and {$p>1$}
small enough so that $\sqrt p \hat{\beta} <1$,
\begin{align*}
\DE_{x}^{\otimes 2}\left[ e^{ \b_N^2\sum_{i=1}^{l_k} \mathbf 1_{S_n^1=S_n^2}} \mathbf{1}_{|S^1_{l_k}| >2Ml_{k}^{1/2}}\right] &\leq \DE_{0}\left[ e^{ p\b_N^2\sum_{i=1}^{l_k} \mathbf 1_{S_{2i} = 0}}\right]^{\frac{1}{p}} \DP_x\left(|S_{l_k}|>2Ml_{k}^{1/2}\right)^{\frac 1 {p'}}\\
&\leq C(\hat \beta) e^{-\frac{c}{p'} M^2},
\end{align*}
for some $c>0$, since $\DE_{0} e^{ \b_N^2\sum_{i=1}^{N} \mathbf 1_{S_{2i} = 0}} = \IE W_N^2 \leq C(\hat \beta)<\infty$ for all $\hat \beta <1$,
{see \eqref{eq-qmomentsfixed}.}
(We have also relied on Hoeffding's inequality to bound the probability, using that $|x|\leq M l_{k}^{1/2}$.)

Combining {\eqref{eq:defAlpha'},
\eqref{eq:biggerthansumWz} and \eqref{eq:sumWz} with}
the two last displays, we obtain that
\begin{align*}
& \DE_x^{\otimes 2}\left[e^{\beta_N^2\sum_{n=1}^{l_k} \mathbf{1}_{S_n^1=S_n^2} }\middle|S^1_{o_k-t} = y_1,S^2_{o_k-t} = y_2\right] \\
& \geq e^{-\theta_\Gamma}\left((1+\varepsilon_N')e^{\lambda^2_k} - 2C(\hat \beta)e^{-\frac{c}{p'} M^2}\right)\\
& = e^{\lambda_k^2} - (1-e^{-\theta_\Gamma})e^{\lambda_k^2} + e^{-\theta_\Gamma}\left(\varepsilon'_N  e^{\lambda_k^2} - 2C(\hat \beta)e^{-\frac{c}{p'}M^2}\right).
\end{align*}
To conclude the proof of the lemma, observe that for all $k\leq K$ we have $\lambda_k^2 \leq \lambda^2$, so that we can choose
\[ \Delta_\Gamma =(1-e^{-\theta_\Gamma}) e^{\lambda^2} + e^{-\theta_\Gamma}\left(|\varepsilon'_N|  e^{\lambda^2} + 2C(\hat \beta)e^{-\frac{c}{p'}M^2}\right),\]
and observe {(using \eqref{eq:defAlpha'})} 
that it satisfies 
${\limsup_{\Gamma'} 
\Delta_\Gamma = 0}$.
\end{proof}

{For technical reasons, we will also need a uniform upper bound on $a_k$.}
\begin{lemma}
  \label{lem-akUB}
  We have 
  \begin{equation}
    \label{eq-akUB}
    {  \sup_{\Gamma}\sup_{k\leq K} a_k\in[1,\infty).}
  \end{equation}
\end{lemma}
\begin{proof} {Since $a_k\geq 1$, the lower bound is trivial. To see the upper bound, we proceed as in the proof of Proposition \ref{prop:ak} and write
as in \eqref{eq:Markovak}:
$$\DE_x^{\otimes 2}\left[e^{\b_N^2\sum_{n=1}^{l_k} \mathbf{1}_{S_n^1=S_n^2} }\mathbf{1}_{S^1_{o_k-t} = y_1} \mathbf{1}_{S^2_{o_k-t} = y_2}\right]
= \sum_{z_1,z_2\in \mathbb Z^2} W(z_1,z_2) \prod_{i=1,2} p_{o_k-t-l_k}(y_i-z_i).$$
We also use the expression in 
\eqref{eq-ratiop} and estimate for $|z_1|\vee |z_2|\leq 2M l_k^{1/2}$ and $x,y_i$ in the ranges appearing in the definition of $a_k$,
\begin{align}
\label{eq-akub1}
&\frac{\bar p_{o_k-t-l_k}(y_1-z_1) \bar p_{o_k-t-l_k}(y_2-z_2)}{\bar p_{o_k-t}(y_1-x) \bar p_{o_k-t}(y_2-x)} \nonumber\\
&= \frac{(o_k-t)^2}{(o_k-t-l_k)^2}e^{-(|y_1|^2+|y_2|^2)((o_k-t)^{-1}-(o_k-t-l_k)^{-1})+\frac{g(\mathbf z,\mathbf y)}{o_k-t-l_k} - \frac{g(\mathbf x,\mathbf y)}{o_k-t}}\\
&{\leq e^{\theta_\Gamma+4/(\delta^2 \nu_2)}}\nonumber.
\end{align}
The estimate of \eqref{eq-akub1} actually extends to the range $\bar z:=|z_1|\vee  |z_2|\leq l_k^{3/5}$ in the form
\begin{align}
\label{eq-akub2}
&\frac{\bar p_{o_k-t-l_k}(y_1-z_1) \bar p_{o_k-t-l_k}(y_2-z_2)}{\bar p_{o_k-t}(y_1-x) \bar p_{o_k-t}(y_2-x)} \nonumber\\
&\leq 2 \frac{(o_k-t)^2}{(o_k-t-l_k)^2}e^{-(|y_1|^2+|y_2|^2)((o_k-t)^{-1}-(o_k-t-l_k)^{-1})+\frac{g(\mathbf z,\mathbf y)}{o_k-t-l_k} - \frac{g(\mathbf x,\mathbf y)}{o_k-t}}\\
&{\leq 2e^{\theta_\Gamma+4/\delta^2 \nu_2} e^{- c \bar z^2/(\nu_2 l_k)}},\nonumber
\end{align}
with $c$ a universal constant; for $\bar z> l_k^{3/5}$, we use a simple large deviations estimate and obtain that
\begin{equation}
\label{eq-akub4}
\frac{\bar p_{o_k-t-l_k}(y_1-z_1) \bar p_{o_k-t-l_k}(y_2-z_2)}{\bar p_{o_k-t}(y_1-x) \bar p_{o_k-t}(y_2-x)} \nonumber\\
\leq e^{-c l_k ^{1/10}}.
\end{equation}
We thus obtain, in analogy with
\eqref{eq:biggerthansumWz},
\begin{align}
\label{eq-akub3}
\DE_x^{\otimes 2}\left[e^{\beta_N^2 \sum_{n=1}^{l_k} \mathbf{1}_{S_n^1=S_n^2} }\middle|S^1_{o_k-t} = y_1,S^2_{o_k-t} = y_2\right] \leq 2e^{\theta_\Gamma +4 /(\delta^2\nu_2)} \sum_{z_1,z_2\in \mathbb Z^2}
W(z_1,z_2)
\end{align}
which, using \cite[Proposition 3.4]{CZ21}, is bounded above by a universal constant depending only on $\hat\beta$.}
\end{proof}
{Recall that $q_0=\lfloor (1-\varepsilon_0)q\rfloor$, see 
\eqref{prop:lowerBoundDN_Upsilon}.}
Our next goal is to show that $\mR_k$ is close to a Poisson 
random variable of parameter $\alpha \binom {{q_0}} 2 / \log N$ by relying on the "two moments suffice" theorem \cite{AGG89}. To verify {the hypothesis of the latter,}
it is more convenient to work with the quantity
\[\tilde \mR_k = \sum_{(i,j)\in \mathcal C_{{q_0}}} \mathbf{1}_{\tau^{(i,j)}_k < \infty}, \quad \tau_k^{(i,j)} = \inf\{n\in \mathcal T_k : S_n^i=S_n^j, n< \sigma^i_k \wedge \sigma^j_k\},\]
(we set $\tau_k^{(i,j)}=\infty$ when the set of the infimum above is empty) which counts the number of \emph{all} the couples that intersect in the time interval $\mathcal T_k$ (whereas $\mR_k$ counts the maximal number of \emph{independent} couples).  
The next proposition states that the law of $\tilde \mR_k$  can be approximated by a Poisson law of mean $\bar \alpha \binom {{q_0}} 2$ and that $\mR_k$ and $\tilde \mR_k$ are close in distribution.
Before stating the proposition, we introduce a few quantities. 
For all $(i,j)\in \mathcal C_{{q_0}}$, we let
$p_{(i,j)} = \DP_{\mathbf x}^{o_k,\mathbf y} (\tau_k^{(i,j)}<\infty)$ and define:
\begin{equation} \label{def:mu}
  \mu = \sum_{(i,j)\in \mathcal C_{{q_0}}} p_{(i,j)}.
\end{equation}
We also set $p_{(i,j),(i',j')} = \DP_{\mathbf x}^{o_k,\mathbf y} (\tau_k^{(i,j)}<\infty, \tau_k^{(i',j')}<\infty)$.
Note that all these quantities depend on $k,\mathbf x,\mathbf y$, but we will show in Section \ref{sec:estimate2moments} that this dependence can be neglected asymptotically. In fact, we prove that $p_{(i,j)}$ can be approximated by $\bar \alpha$ and that $\mu$ can be approximated by $\bar \alpha \binom {{q_0}} 2$.

\begin{proposition} \label{prop:2moments} There exists $\Delta_{\Gamma,{\ref{prop:2moments}}} >0$ such that
  {$\limsup_{\Gamma'}\Delta_{\Gamma,\ref{prop:2moments}}= 0$,} 
  for which $\varepsilon_N^\star = q^3\bar \alpha(1+\Delta_{\Gamma,{\ref{prop:2moments}}})(\bar \alpha + \frac{\log \log N}{\gamma \log N})$ satisfies:
\begin{equation} \label{eq:2momentsSufficeTV}
  \sup_{k\leq K} \sup_{\substack{\mathbf x\in B_{{q_0},k},\mathbf y\in 
  B_{{q_0},k+1}\\ \mathbf x \sim_{o_k} \mathbf y}} \mathrm{d_{TV}} \left|\DP_{\mathbf x}^{o_k,\mathbf y}\left(\tilde \mR_k = \cdot\right) - \mathcal{P}(\mu) \right| \leq C\varepsilon_N^\star,
\end{equation}
and 
\begin{equation} \label{eq:R_kandtildeR_k}
  \sup_{k\leq K} \sup_{\substack{\mathbf x\in B_{{q_0},k},\mathbf y\in 
  B_{{q_0},k+1}\\ \mathbf x \sim_{o_k} \mathbf y}} \mathrm{d_{TV}} \left|\DP_{\mathbf x}^{o_k,\mathbf y}(\mR_k = \cdot ) - \DP_{\mathbf x}^{o_k,\mathbf y}\left(\tilde \mR_k = \cdot\right)\right|\leq  C\varepsilon_N^\star,
  \end{equation}
  where $\mathrm{d_{TV}}|\cdot|$ denotes the distance in total variation and $\mathcal P(\mu)$ is the Poisson distribution of mean $\mu$ {from \eqref{def:mu}.}
\end{proposition}
\begin{remark} \label{rk:epsilonStar}
Since $q^2=O(\log N)$ we have $\limsup_N \varepsilon^\star_N = 0$.
\end{remark}

\begin{proof}
  We first prove \eqref{eq:2momentsSufficeTV}.
  Following \cite{AGG89},  we define $B_{(i,j)} = \{(i',j')\in\mathcal C_{{q_0}}: \{i',j'\}\cap \{i,j\}\neq \emptyset\}$ and
\begin{align*}
  &e_1 = \sum_{(i,j)\in\mathcal C_{{q_0}}} \sum_{(k,l)\in B_{(i,j)}} p_{(i,j)} p_{(k,l)},\\
  &e_2 = \sum_{(i,j)\in\mathcal C_{{q_0}}} \sum_{(i',j')\in B_{(i,j)}\setminus\{(i,j)\}} p_{(i,j),(i',j')}.
\end{align*}
By Proposition \ref{prop:pxy} and Proposition \ref{prop:3intersections}, we have
$e_1 \leq Cq^3  (1+\Delta_{{\Gamma},\ref{prop:pxy}})^2 \bar \alpha^2 $ and
$e_2 \leq C (1+\Delta_{\Gamma,\ref{prop:3intersections}}) q^3 \bar \alpha \frac{\log \log N}{\gamma \log N} $ with $\limsup_{{\Gamma'}} \Delta_{\Gamma} = 0$ for both errors. 
We then obtain \eqref{eq:2momentsSufficeTV} by applying \cite[Theorem 1]{AGG89}. 

We turn to \eqref{eq:R_kandtildeR_k}.
By a standard property of the distance in total variation,
\[\mathrm{d_{TV}} \left| \DP_{\mathbf x}^{o_k,\mathbf y}(\mR_k = \cdot ) - \DP_{\mathbf x}^{o_k,\mathbf y}(\tilde \mR_k = \cdot )\right| \leq 2 \DP_{\mathbf x}^{o_k,\mathbf y}(\mR_k\neq \tilde \mR_k).\]
Then, observe that on the event $\{\mR_k\neq \tilde \mR_k\}$, there 
 exist two couples $(i,j),(i',j')\in  \mathcal C_{{q_0}}$ such that $|\{i,j\}\cap\{i',j'\}| = 1$ with $\tau_k(i,j) < \infty$ and $\tau_k(i',j') < \infty$. (See also Remark \ref{rk:maximalityOfRk}). Hence $\DP_{\mathbf x}^{o_k,\mathbf y}(\mR_k\neq \tilde \mR_k) \leq e_2$, which gives \eqref{eq:R_kandtildeR_k}.
\end{proof}

In the following proposition, we use a certain constant $\Delta_{\Gamma,\ref{prop:pxy}}>0$ introduced below in Proposition \ref{prop:pxy}, and which satisfies $\limsup_{{\Gamma'}} \Delta_{\Gamma,\ref{prop:pxy}} =  0$.
\begin{proposition} \label{prop:akRk} There exist $c>0$, $\alpha_0>0$ and $N_0=N_0(\tilde{\Gamma})$ such that for all $\alpha>\alpha_0$ and $N\geq N_0$, we have for all $k\leq K$,
\begin{equation} \label{eq:1minusDelta}
  \inf_{\substack{\mathbf x\in B_{{q_0},k},\mathbf y\in B_{{q_0},k+1}\\ \mathbf x \sim_{o_k} \mathbf y}} \DE_{\mathbf x}^{o_k,\mathbf y}\left[ a_k^{\mR_k}\right] \geq e^{\binom{{q_0}}{2}\bar \alpha (a_k-1)(1-\Delta_{\Gamma,\ref{prop:pxy}}) }\left(1- \Delta'_{\Gamma,{\ref{prop:akRk}}} \right),
\end{equation}
where 
$\Delta'_{\Gamma,\ref{prop:akRk}} \in [0,\frac{1}{2}]$ satisfies 
$\limsup_{{\Gamma'}} \binom{q}{2}^{-1} K \Delta'_{\Gamma,\ref{prop:akRk}} = 0$.
\end{proposition}
\begin{proof}
  Let $\mathcal R$ be distributed as $\mathcal P(\mu)$ and {recall}
  $\varepsilon_N^\star$ from  Proposition \ref{prop:2moments}. For all $r_0\in \mathbb N$, we have
\begin{equation} \label{eq:toDeltaprime} 
\begin{aligned} \DE_{\mathbf x}^{o_k,\mathbf y}\left[ a_k^{\mR_k}\right] & \geq \DE_{\mathbf x}^{o_k,\mathbf y}\left[ a_k^{\mR_k} \mathbf{1}_{\mR_k \leq r_0} \right] \\
  & \geq \DE \left[a_k^{\mathcal R} \mathbf{1}_{\mathcal R\leq r_0}\right] - a_k^{r_0} \varepsilon_N^\star\\
  & \geq e^{\mu(a_k-1)}\left(1- \frac{\mu^{r_0+1}}{(r_0+1)!} - a_k^{r_0} \varepsilon_N^\star\right),
\end{aligned}
\end{equation}
where we have used that $\DE[a^{\mathcal R}] = e^{\mu(a-1)}$, that  $\DE[a^{\mathcal R} \mathrm{1}_{\mathcal R \geq r}] \leq e^{\mu(a-1)} \frac{\mu^r}{r!}$ for all $a>0,r\in\mathbb N$ and that $e^{\mu(a_k-1)}\geq 1$.

Recall  the constants  $\Delta_{\Gamma}=\Delta_{\Gamma,\ref{prop:pxy}}>0$ from Proposition \ref{prop:pxy},
which satisfy $\limsup_{\Gamma'} \Delta_{\Gamma} = 0$. Uniformly on $\mathbf x,\mathbf y,k$, we have:
\begin{equation} \label{eq:approx_mu}
  \left|\mu-\binom{{q_0}}{2} \bar \alpha \right| \leq \binom{{q_0}}{2} \bar \alpha \Delta_{\Gamma}.
\end{equation}
Next, define 
$c:=\sup_{k} a_k \in (1,\infty)$ {by \eqref{eq-akUB},}
and 
\[
\Delta'_{\Gamma',\ref{prop:akRk}}=\Delta_\Gamma' := \inf_{r_0 \in \mathbb N} \left\{ \frac{\left(\binom {{q_0}} 2 \bar \alpha (1+\Delta_\Gamma)\right)^{r_0+1}}{(r_0+1)!} + c^{r_0}  \varepsilon^\star_N\right\}.
\]
We first show that 
 $\limsup_{\Gamma'} \binom{{q_0}}{2}^{-1} K \Delta'_\Gamma = 0$ and then that $\Delta'_\Gamma \in [0,\frac{1}{2}]$ for $\alpha$ and $N$ large enough. Together with \eqref{eq:toDeltaprime} and \eqref{eq:approx_mu}, this yields the proposition. 
Since $K\leq \bar \alpha^{-1} \log \gamma^{-1}$ by \eqref{eq:K}, we have for all $r_0\in \mathbb N$,
\begin{align*}
  \binom{q}{2}^{-1} K \frac{\left(\binom {{q_0}} 2 \bar \alpha (1+\Delta_\Gamma)\right)^{r_0+1}}{(r_0+1)!} \leq (\log \gamma^{-1})(1+\Delta_\Gamma) \frac{\left(\binom 
{{q}} 2 \bar \alpha (1+\Delta_\Gamma)\right)^{r_0}}{(r_0+1)!}.
\end{align*}
Moreover, $\limsup_{N} \binom q 2 \bar \alpha \leq C_0\alpha$ with $C_0 \in (0,\infty)$ by hypothesis. Hence, if we define $\Delta_{\tilde \Gamma} = \limsup_N \Delta_{\Gamma}$, the supremum limit over ${\Gamma'}$ of the right hand side of the last display is less than
\[
  \limsup_{{\widehat  \Gamma}} \left\{(\log \gamma^{-1})(1+\Delta_{\tilde \Gamma}) \frac{\left(C_0\alpha (1+\Delta_{\tilde \Gamma})\right)^{r_0}}{(r_0+1)!}\right\}.
\]
 (Recall that $\Gamma'=(N,\hat{\Gamma})$.)
If we choose $r_0 =  \lceil e^2 C_0 \alpha (1+\Delta_{\tilde \Gamma})\rceil$ and use Stirling’s approximation $r! \geq (r/e)^r$ valid for all $r\in \mathbb N$, we find that the last display is smaller than 
\begin{equation*} 
  \limsup_{{\widehat \Gamma}} \left\{(\log \gamma^{-1})(1+\Delta_{\tilde \Gamma}) e^{- e^2 C_0 \alpha }\right\} = 0,
\end{equation*}
where the equality holds since we take the limit $\alpha \to \infty$ {with $\gamma$ fixed}.
Hence, by choosing $r_0 =  \lceil e^2 C_0 \alpha (1+\Delta_{\tilde \Gamma})\rceil$ we have shown that
\begin{equation} \label{eq:limsupKDelta}
  \limsup_{ \Gamma'} \binom q 2^{-1} K \Delta_\Gamma' \leq \limsup_{ \Gamma'}  \left\{ \binom q 2^{-1} K c^{\lceil e^2 C_0 \alpha (1+\Delta_{\tilde \Gamma})\rceil} \varepsilon_N^\star\right\}.
\end{equation}
We now prove that the last {limit superior vanishes}. 
By definition of $\varepsilon^\star_N$ in Proposition \ref{prop:2moments}, we have
\[\binom q 2 ^{-1} K\varepsilon^\star_N \leq Cq \bar \alpha^{-1} \log \gamma^{-1}  \bar \alpha (1+\Delta_\Gamma)\left(\frac{\alpha}{\log N} + \frac{\log \log N}{\gamma \log N}\right),\] 
Using that $\limsup_N q^2/ \log N < \infty$, we obtain that $\limsup_N \binom q 2^{-1} K\varepsilon_N^\star = 0$ and thus $\limsup_{{\Gamma'}} \binom q 2^{-1} K \Delta_\Gamma' = 0$ by \eqref{eq:limsupKDelta}.

To conclude, we prove that $\Delta'_\Gamma \leq 1/2$. If we choose again $r_0 =  \lceil e^2 C_0 \alpha (1+\Delta_{\tilde \Gamma})\rceil$, we find using Stirling’s approximation as before that
$\limsup_N \Delta_\Gamma' \leq e^{- e^2 C_0  \alpha}$. 
So if we let $\alpha$ large enough followed by $N$ large enough (depending on $\tilde \Gamma$) we obtain that $\Delta_\Gamma' \leq  1/2$.
\end{proof}

Here is our last technical estimate. Recall the definition of $D_N$ in Proposition \ref{prop:lowerBoundDN_Upsilon}.
\begin{proposition} \label{prop:D_N}
There exist $c,c'>0$ such that 
\begin{equation}\label{eq:lowerboundDNProp} D_N \geq 1- Ke^{\varepsilon_0 q(c'  \log \varepsilon_0^{-1} - c \delta^{-2})}.
\end{equation}
\end{proposition}
\begin{proof} Define $H^i_{k} = \{S_{L_k}^i \notin B(0,\delta^{-1} L_k^{1/2}) \text{ or } S_{L_{k+1}}^i \notin B(0,\delta^{-1} L_{k+1}^{1/2})\}$. By definition of $D_N$ and the union bound, 
\begin{equation}\label{eq:lowerBoundDN}
D_N \geq 1-\sum_{k=1}^K \DP_0^{\otimes q}(A_k^c).
\end{equation} 
Let $p = \lfloor \varepsilon_0 q \rfloor$. The event $A_k^c$ implies that there exists $i_1<\dots<i_p \leq q$ such that $H^{i_r}_k$ holds for all $r\leq p$. Hence, by independence of the walks,
\[\DP_0^{\otimes q}(A_k^c) \leq \binom{q}{p} \DP_0(H^1_k)^p.\]
By Hoeffding's inequality there exists $c>0$ such that $\DP_0(H_k^1) \leq e^{-c\delta^{-2}}$. Since $\varepsilon_0$ is small, we further have that $\binom q p \leq e^{c'\varepsilon_0 q \log \varepsilon_0^{-1}}$ for some $c'>0$ via Stirling’s approximation.
\end{proof}

\subsection{Proof of Theorem \ref{th:main}}
By Proposition \ref{prop:lowerBoundDN_Upsilon}, we have
\begin{equation}
 \binom q 2^{-1} \log \IE\left[ W_N^q\right]  \geq \binom q 2 ^{-1}\log D_N + \binom q 2^{-1}\sum_{k=1}^K \log \Upsilon_k.
\end{equation}
We first observe that 
\begin{equation} \label{eq:limsupGammaDN}
  \limsup_\Gamma \binom q 2^{-1} (-\log D_N) = 0.
\end{equation}
Since $\log \log N =  O(q)$, we can find $c_0>0$ such that $q \geq c_0 \log \log N$ for $N$ large enough. Now, because we take the limit $\delta\to 0$ before $\varepsilon_0\to 0$, we can assume that in \eqref{eq:lowerboundDNProp} we have $\varepsilon_0(c'  \log \varepsilon_0^{-1} - c \delta^{-2}) < -2c_0^{-1}$, so that using \eqref{eq:K} we have $D_N \geq 1-\log  \gamma^{-1} \alpha^{-1} \log Ne^{-2\log \log N}$ which converges to 1 as $N\to\infty$. This gives \eqref{eq:limsupGammaDN}.

Next, by Proposition \ref{prop:iteration} and Proposition \ref{prop:akRk},
\begin{equation} \label{eq:LBsumlogUpsilon}
  \binom q 2^{-1}\sum_{k=1}^K \log \Upsilon_k \geq \binom q 2^{-1} \binom {{q_0}} 2  (1-{\Delta_{\Gamma,\ref{prop:pxy}}}) \bar \alpha \sum_{k=1}^K   (a_k-1) +\binom q 2^{-1} K\log(1-{\Delta'_{\Gamma,\ref{prop:akRk}}}).
\end{equation}
Since ${\Delta'_{\Gamma,\ref{prop:akRk}}} \leq 1/2$, we have that $-\binom q 2^{-1} K\log(1-{\Delta'_{\Gamma,\ref{prop:akRk}}}) \leq C\binom{q}{2}^{-1} K
{\Delta'_{\Gamma,\ref{prop:akRk}}}$. Hence by the definition of ${\Delta'_{\Gamma,\ref{prop:akRk}}}$ {and \eqref{eq:K}},
$\limsup_\Gamma \binom q 2 ^{-1} K\left(-\log(1-{\Delta'_{\Gamma,\ref{prop:akRk}}})\right) = 0$. This deals with the second term of the right-hand side of \eqref{eq:LBsumlogUpsilon}.
Concerning the first term, we will show that
\begin{equation} \label{eq:LB_MainTermInMainProof}
\liminf_\Gamma \binom q 2^{-1} \binom {{q_0}} 2  (1-\Delta_{\Gamma,\ref{prop:pxy}}) \bar \alpha \sum_{k=1}^K   (a_k-1) \geq \lambda(\hat \beta)^2.
\end{equation}
First, we rely on Proposition \ref{prop:ak} to find that
\[\bar \alpha \sum_{k=1}^K (a_k - 1) \geq \bar \alpha \sum_{k=1}^K \left(e^{\lambda_k^2}-1\right) - \bar \alpha K {\Delta_{\Gamma,\ref{prop:ak}}},\]
where $\limsup_{\Gamma} \bar \alpha K |{\Delta_{\Gamma,\ref{prop:ak}}}| \leq \limsup_{\Gamma} \log \gamma^{-1}|{\Delta_{\Gamma,\ref{prop:ak}}}| = 0$ by \eqref{eq:K} and the definition of ${\Delta_{\Gamma,\ref{prop:ak}}}$. Now, recalling the definition of $\lambda_k^2$ in \eqref{eq:def_lambdak2}, observe that
\begin{align*}
  \bar \alpha \sum_{k=1}^K \left(e^{\lambda_k^2}-1\right) & \geq \bar \alpha \sum_{k=1}^K  \frac{\hat \beta^2 \gamma e^{k\bar \alpha}}{1-\hat \beta^2 \gamma e^{k\bar \alpha}}.
\end{align*}  
Therefore, by Riemann sum approximation and the lower bound on $K$ in \eqref{eq:K} (recall that ${q_0}=\lfloor (1-\varepsilon_0)q\rfloor$):
\begin{align*}
  &\liminf_N \binom q 2^{-1} \binom {{q_0}} 2 \bar \alpha \sum_{k=1}^K \left(e^{\lambda_k^2}-1\right)\\
  & \geq  (1-\varepsilon_0)^2 \int_{0}^{\log \gamma^{-1}} \frac{\hat \beta^2 \gamma e^x}{1-\hat \beta^2 \gamma e^x} \dd x = (1-\varepsilon_0)^2 \left( \log(1-\gamma \hat \beta^2)- \log(1- \hat \beta^2)\right),
\end{align*}
where the last quantity converges to $\lambda(\hat \beta^2)$ as $\gamma,\varepsilon_0\to 0$. This gives \eqref{eq:LB_MainTermInMainProof}.

Putting everything together yields the lower bound
$\liminf_\Gamma \binom q 2^{-1} \log \IE\left[ W_N^q\right] \geq \lambda(\hat \beta)^2$, that is $\liminf_N \binom q 2^{-1} \log \IE\left[ W_N^q\right]  \geq \lambda(\hat \beta)^2$.
\qed

\subsection{Proof of Theorem \ref{th:qlarge}}\label{sec:qlarge}
  Introduce the event 
  \[ \mathcal{A}=\{ S_{2k}^i= 0, k=0,\ldots, \lfloor N/2\rfloor, i=1,\ldots,q\}.\]
  Note that $\DP(\mathcal{A})\geq (1/4)^{{q}\lfloor N/2 \rfloor}$. On the event $\mathcal{A}$ we
  have a total of at least $(N/2)\binom q 2$ intersections. Substituting in
  \eqref{eq:formula_moments} then yields that
  \[ \IE[W_N^q]\geq e^{\beta_N^2 (N/2)\binom q 2} (1/4)^{{q}\lfloor N/2 \rfloor}.\]
  This gives the result. \qed

\section{Estimates for ``two moments suffice''} \label{sec:estimate2moments}
\subsection{Two-particle intersection probability}
The goal of this section is to give an estimate on $p_{(i,j)} = \DP_{\mathbf x}^{o_k,\mathbf y} (\tau_k^{(i,j)}<\infty)$ used in the proof of Proposition \ref{prop:2moments}. 
To simplify future notations, we write $\tau_k=\tau_k^{(1,2)}$ and $p_{\mathbf w,\mathbf z}=\DP_{\mathbf w}^{o_k,\mathbf z}(\tau_k<\infty)$ for $\mathbf w\sim_{o_k} \mathbf z\in \mathbb Z^2 \times \mathbb Z^2$. The following proposition provides the desired asymptotics. (Note that $p_{(i,j)} = p_{(x_i,x_j),(y_i,y_j)}$).
\begin{proposition} \label{prop:pxy} There exists $\Delta_{\Gamma,\ref{prop:pxy}}>0$ such that $\limsup_{{\Gamma'}} \Delta_{\Gamma,\ref{prop:pxy}} = 0$ and
\begin{equation} \label{eq:pxy}
\sup_{k\leq K} \sup_{\substack{\mathbf x\in B_{2,k},\mathbf y\in B_{2,k+1}\\ \mathbf x \sim_{o_k} \mathbf y}} \left|p_{\mathbf x,\mathbf y} - \bar \alpha \right| \leq  \bar \alpha \Delta_{\Gamma,\ref{prop:pxy}}.
\end{equation}
\end{proposition}
We will prove \eqref{eq:pxy} using a sequence of lemmas (we refer to the end of the section for the proof of Proposition \ref{prop:pxy}).
As a first step, we show that $p_{\mathbf x,\mathbf y}$ can be replaced by $p_{\mathbf x} = \DP_{\mathbf x}^{\otimes 2} (\tau_k<\infty)$, i.e.\ $p_{\mathbf x}$ is defined as $p_{\mathbf x,\mathbf y}$ except there is no conditioning on the endpoint.
\begin{lemma} \label{lem:removeConditioning} There exists $\Delta_{\Gamma,\ref{lem:removeConditioning}} >0$ satisfying $\limsup_{{\Gamma'}} \Delta_\Gamma = 0$ such that for all $k\leq K$ and all $\mathbf x\in B_{2,k}$,
\begin{equation} \label{eq:removeCondi}
  \sup_{\substack{\mathbf y\in B_{2,k+1}\\ \mathbf y \sim_{o_k} \mathbf x}} \left|p_{\mathbf x,\mathbf y} - p_{\mathbf x} \right| \leq p_{\mathbf x} \Delta_{\Gamma,\ref{lem:removeConditioning}}.
\end{equation}
\end{lemma}
\begin{proof}
By Markov's property, we have
\begin{align*}
p_{\mathbf x,\mathbf y}-p_{\mathbf x} &= \DE_{\mathbf x}^{\otimes 2}\left[\mathbf{1}_{\tau_k<\infty} \mathbf{1}_{S_{o_k}^1 = y_1,S_{o_k}^2 = y_2}\right]\DP_{\mathbf x}^{\otimes 2}(S_{o_k}^1 = y_1,S_{o_k}^2 = y_2)^{-1} - p_{\mathbf x}\\
& = \DE_{\mathbf x}^{\otimes 2}\left[\mathbf{1}_{\tau_k<\infty} \left(\frac{\DP_{S^1_{\tau_k},S^2_{\tau_k}}^{\otimes 2}(S^1_{o_k-\tau_k}=y_1,S^2_{o_k-\tau_k}=y_2)}{\DP_{\mathbf x}^{\otimes 2}(S_{o_k}^1 = y_1,S_{o_k}^2 = y_2)}-1\right)\right].
\end{align*}
Define:
 \[V_k = \sup_{\substack{z\in B(0,M l_k^{1/2})\\z\sim_{(o_k-t)}y_1,y_2}} \sup_{t\in \mathcal T_k} \left| \frac{p_{o_k-t}(y_1-z)p_{o_k-t}(y_2-z)}{p_{o_k}(y_1-x_1) p_{o_k}(y_2-x_2)} - 1 \right|.\]
Since by definition $S^1_{\tau_k},S^2_{\tau_k}\in B(0,M l_k^{1/2})$ when $\tau_k<\infty$, we have $|p_{\mathbf x,\mathbf y}-p_{\mathbf x}| \leq p_{\mathbf x} V_k$. It is thus enough to prove that
\begin{equation} \label{eq:equationVk}
  V_k\leq C e^{|\varepsilon_N| + b_0+b_1} \left(|\varepsilon_N| + b_0 + b_1 + 12\nu_2^{-1}\right)=:\Delta_{\Gamma,\ref{lem:removeConditioning}},
\end{equation} 
where $\varepsilon_N = \varepsilon_N(\tilde \Gamma) \to 0$ as $N\to\infty$ and
\begin{equation} \label{eq:defb_0b_1}
 b_0 = 8 \nu_2^{-1} \delta^{-2},\quad b_1 = 4\left( \delta^{-1} M \nu_2^{-1/2} + M^2 \nu_2^{-1} +2 \delta^{-2} e^{-\frac{1}{2}\gamma \alpha/4 } \right),
\end{equation}
since then $\limsup_{\Gamma'}\Delta_{\Gamma,\ref{lem:removeConditioning}}=0$.
Similarly to the proof of Proposition \ref{prop:ak}, the argument leading to \eqref{eq:equationVk} relies on the local central limit theorem. In the following we assume that $z\in B(0,Ml_{k}^{1/2})$, $\mathbf x\in B_{2,k}$ and $\mathbf y \in B_{2,k+1}$.
We first note that $o_k-t \geq \nu_2 l_k$. By \eqref{eq:boundl_k}-(ii), it further holds that $|y_i-z| \leq (2\delta^{-1} \nu_2^{1/2} + M)l_k^{1/2}$ and $|y_i-x_i| \leq 4\delta^{-1}\nu_2^{1/2}l_k^{1/2}$. Hence $|y_i-z| \leq c_N(o_k-t)$ and $|y_i-x_i|\leq c_No_k$ with $c_N\to 0$, so Theorem \ref{AppendixLCLT} gives:
\begin{align*}
&p_{o_k-t}(y_i-z) = 2 \bar{p}_{o_k-t}(y_i-z) e^{O(d_k)},\\
&p_{o_k}(y_i-x_i) = 2 \bar{p}_{o_k}(y_i-x_i) e^{O(d_k)},
\end{align*}
  where $\bar p_{s}(x) = \frac{1}{\pi s} e^{-|x|^2/s}$ and $d_k = \frac{1}{\nu_2 l_k} + \frac{\delta^{-4} \nu_2^2 + M^4}{\nu_2 ^3l_k}\leq c N^{-\gamma}$ with $c=c(\delta,\nu_2,M)$. We now come back to $V_k$. Letting $\mathbf z = (z,z)$, we find that
\begin{align*}
\frac{\bar p_{o_k-t}(y_1-z) \bar p_{o_k-t}(y_2-z)}{\bar p_{o_k}(y_1-x_1) \bar p_{o_k}(y_2-x_2)} = \frac{o_k^2}{(o_k-t)^2}e^{-(|y_1|^2+|y_2|^2)(o_k^{-1}-(o_k-t)^{-1})+\frac{g(\mathbf z,\mathbf y)}{o_k-t} -2 \frac{g(\mathbf x,\mathbf y)}{o_k}},
\end{align*}
where $g(\mathbf x,\mathbf y) = 2\langle y_1,x_1\rangle + 2\langle y_2,x_2\rangle - |x_1|^2 - |x_2|^2$. Recall $b_0$ and $b_1$ in \eqref{eq:defb_0b_1}.
The first term in the exponential above is positive and smaller than
\[\left(|y_1|^2+|y_2|^2\right)\frac{t}{o_k(o_k-t)} \leq \frac{\delta^{-2} L_{k+1}(\nu_1 l_{k-1} + l_k)}{\nu_2^2 l_k^2}\leq b_0,\]
by \eqref{eq:boundl_k}-(ii) and Remark \ref{rem:2lk}.
The sum of the absolute values of the two other terms in the exponential is smaller than
\[(\nu_2 l_k)^{-1}\left(4 \delta^{-1} L_{k+1}^{1/2} M l_k^{1/2} + 2 M^2 l_k + 4\delta^{-2} L_{k+1}^{1/2} L_k^{1/2} + 2\delta^{-2} L_k\right) \leq b_1,\]
by the Cauchy-Schwarz inequality and \eqref{eq:boundl_k}(i),(ii).
Moreover,
\[\left|\frac{o_k^2}{(o_k-t)^2} -1\right| = \frac{t(2o_k-t)}{(o_k-t)^2}\leq \frac{(2l_k)(6\nu_2 l_k)}{(\nu_2 l_k)^2}\leq 12 \nu_2^{-1}.\]
Combining these estimates entails \eqref{eq:equationVk} using that $|e^x - 1|\leq |x|e^{|x|}$ for all $x\in\mathbb R$.
\end{proof}

Next, we show that we can neglect the condition $n< \sigma^1_k\wedge \sigma^2_k$ in the definition of $\tau_k= \tau_k^{1,2}$. Thus, we define:
\begin{equation}
\label{eq-tautildep}
\tilde \tau_k = \inf\left\{n\in \mathcal T_k| S_n^1=S_n^2\right\} \qquad \textrm{and} \qquad \tilde p_{\mathbf x}=\DP_{\mathbf x}^{\otimes 2}(\tilde \tau_k < \infty). 
\end{equation}
\begin{lemma} \label{lem:remove_barrier}
There exists $c>0$ such that 
\begin{equation} \label{eq:remplacerSansBarriere}
\sup_{k\leq K} \sup_{\mathbf x \in B_{2,k}} |p_{\mathbf x} - \tilde p_{\mathbf x}| \leq C\frac{e^{-c {M^2}}}{\gamma \log N}  .
\end{equation}
\end{lemma}
\begin{proof}
We have:
\begin{align*}
p_{\mathbf x} &= \DP_{\mathbf x}^{\otimes 2}(\tilde \tau_k < \infty,\tilde \tau_k < \sigma^1_k\wedge \sigma^2_k) =  \DP_{\mathbf x}^{\otimes 2}(\tilde \tau_k < \infty) - \DP_{\mathbf x}^{\otimes 2}(\tilde \tau_k < \infty,\sigma^1_k \wedge \sigma^2_k \leq \tilde \tau_k),
\end{align*}
hence, by the union bound,
\[|p_{\mathbf x}-\tilde p_{\mathbf x}|\leq \sum_{i=1,2} \DP_{\mathbf x}^{\otimes 2}(\tilde \tau_k < \infty,\sigma^i_k \leq \tilde \tau_k).\]
We will bound from above the term corresponding to $i=1$ in the sum. The other term is treated the same way. Since $\mathcal T_k = \llbracket \nu_1 l_{k-1},\nu_1 l_{k-1} + l_k \rrbracket$,
\[\DP_{\mathbf x}^{\otimes 2}(\tilde \tau_k < \infty,\sigma^1_k \leq \tilde \tau_k) \leq  \DP_{\mathbf x}^{\otimes 2}\left(\sigma^1_k \leq \tilde \tau_k,\exists n \in \llbracket \sigma^1_k,\sigma^1_k+ l_{k} \rrbracket : S_n^1=S_n^2\right),\]
hence by Markov's property, 
\begin{align}
\DP_{\mathbf x}^{\otimes 2}(\tilde \tau_k < \infty,\sigma^1_k \leq \tilde \tau_k) \nonumber & \leq \sum_{m\in \mathcal T_k} \DE_{\mathbf x}^{\otimes 2}\left[ \mathbf{1}_{\sigma^1_k = m} \DP_{S^1_m,S^2_m}^{\otimes 2}\left(\exists n \leq l_k : S_n^1=S_n^2\right)\right] \nonumber\\
& = \sum_{m\in \mathcal T_k} \sum_{x,y\in \mathbb Z^2} \DE_{x_1}\left[ \mathbf{1}_{\sigma^1_k = m} \mathbf{1}_{S_m^1=x}\right] p_{m}(y-x_2) h_k(x-y), \label{eq:sumToSplit}
\end{align}
where $h_k(z) = \DP_{z}(\exists n\leq l_k:S_n = 0)$. 
It {follows} from \cite[Th\'{e}or\`{e}me 3.6]{LeGall86} that 
\[(\log l_k)h_k(z) \leq C \left(\log \left\{{l_k} |z|^{-2}\right\}\right)_+ + C \mathbf{1}_{|z|^2\geq l_k}.\]
We thus split the sum that appears in \eqref{eq:sumToSplit} into $Q_1 + Q_2$, where $Q_1$ contains the terms for which $|x-y|^2\geq l_k$. Then
 $Q_1\leq C(\log l_k)^{-1} \DP_{x_1}\left(\sigma_k^1 \in \mathcal T_k\right)$, where by \eqref{eq:boundl_k} and \eqref{eq:assumptionOnParameters}-(iv), we have
\begin{equation}\label{eq:boundx_1M}
|x_1| \leq \delta^{-1} L_k^{1/2} \leq 2\nu_2^{1/2} e^{-\frac{1}{2}\gamma \alpha/2} l_{k}^{1/2}\leq  \frac{M}{2} l_k^{1/2},
\end{equation}
so that,
\begin{align*}
Q_1 \leq \frac{C}{\log l_k} \DP_{0}\left(\sup_{n\in \mathcal T_k} |S_n| \geq M l_k^{1/2} - |x_1| \right) & \leq \frac{C}{\log l_k} e^{-c \frac{\left(M l_k^{1/2} - |x_1|\right)^2}{ \nu_1 l_{k-1} + l_k}} \leq C \frac{ e^{-c\frac{M^2}{8}}}{\log l_k},
\end{align*}
for some $c>0$, by Doob's inequality and Hoeffding's lemma. (Note that for the last inequality, we have used Remark \ref{rem:2lk}). Then, 
\begin{equation} \label{eq:defQ2}
Q_2\leq \frac{C}{\log l_k}\sum_{m\in \mathcal T_k} \sum_{x\in \mathbb Z^2} \DE_{x_1}\left[ \mathbf{1}_{\sigma^1_k = m}   \mathbf{1}_{S_m^1=x}\right] A_m(x),
\end{equation}
with
\begin{align*}
A_m(x) = \sum_{\substack{y\in \mathbb Z^2, y\neq x,\\|x-y|^2< l_k}} p_{m}(y-x_2) \log \frac{l_k}{|x-y|^2}.
\end{align*}
Since $m=\sigma_k^1$ implies that $S^1_m$ lies outside the ball $B(0,Ml_{k}^{1/2})$, we can restrict the sum in \eqref{eq:defQ2} to $|x|> Ml_{k}^{1/2}$. 
Then, as $|x_2|$ satisfies the same bound as $|x_1|$ in \eqref{eq:boundx_1M}, we get that $|x-x_2|\geq \frac{M}{2} l_{k}^{1/2}$, which implies that $|y-x_2| \geq \frac{M}{4} l_k^{1/2}$ under the condition $|x-y|^2 < l_k$.  Thus, given that 
$m\geq \nu_1 l_{k-1}$, we can apply the local limit theorem (Theorem \ref{th:LCLT})  to obtain that:
\[A_m(x) \leq \frac{C}{m} \sum_{\substack{z\in \mathbb Z^2,  \\ 0<|z|^2< l_k}} e^{-\frac{M^2}{16}\frac{l_k}{m}}  \log \frac{l_k}{|z|^2},\]
 and hence
\[A_m(x) \leq \frac{C}{m} e^{-\frac{M^2}{16} \frac{l_k}{m}} \sum_{r=1}^{\lfloor l_k^{1/2}\rfloor} r \log \frac{l_k}{r^2} \leq C\frac{l_k}{m}e^{-\frac{M^2}{16}\frac{l_k}{m}} \leq C e^{-\frac{M^2}{32} \frac{l_k}{m}},\]
where in the second inequality, we have used a comparison to an integral where  $C\int_{0}^1 u\log u^{-2} \dd u < \infty$. 
Using that in the last exponential term we have $m\leq 2 l_{k}$, we get via \eqref{eq:defQ2} that $Q_2 \leq \frac{C}{\log l_k} e^{-\frac{M^2}{64}}$.
This gives \eqref{eq:remplacerSansBarriere} since $l_k \geq N^{\gamma}$.
%
\end{proof}

We introduce the shorthand notation
$g_k(x) = \DP_{x}(\exists n \in \mathcal T_k : S_n = 0)$ that satisfies
\begin{equation} \label{eq:identityptog}
\tilde p_{\mathbf x} = g_k(x_1-x_2),
\end{equation} 
{recall \eqref{eq-tautildep}.}
Our aim is to use the KMT coupling (see \cite{Zaitsev98} and references therein) to estimate $g_k(x)$. The KMT coupling ensures that one can couple, {with high probability,} the simple random walk $(S_n)$ to a standard $2$-dimensional Brownian motion $(B_t)$  with an error term $\Delta_n = \max_{t\leq n} | S_t- \sigma B_t|$ (with $\sigma^2 =1/2$) satisfying $\Delta_n = O(\log n)$. 
We will use the coupling to compare the hitting time of $0$ of the random walk to the entry time of Brownian motion in a ball of radius $c\log N$. This will turn out helpful as there are good estimates by Spitzer \cite{Spitzer58} on the probability of the last event.\footnote{There exist similar estimates for the random walk itself, such as \cite{RR66}, but unfortunately they are not sharp enough to estimate $g_k(x)$ directly in our context.} 
  
Let $t_1 = \nu_1 l_{k-1}$ and $t_2 = t_1 + l_k$ denote the boundaries of $\mathcal T_k$ and $t_2' = t_1 + l_k/2 $. We define:
\begin{equation} \label{eq:defG_k}
  G_{k,c_0} = \left\{\inf_{t\in [t_1,t_2]} |B_t| \leq c_0\log N\right\} \text{ and } G_{k,c_0}' = \left\{\inf_{t\in [t_1,t_2']} |B_t| \leq c_0\log N\right\}.
\end{equation}
\begin{lemma}\label{lem:KMT}
 There exists $c_0,c_1,c_1',c_2,c_\gamma>0$ such that for all $x\in B\left(0,2\delta^{-1} L_k^{1/2}\right)$ and all $k\leq K$,
\begin{equation} \label{eq:upperBoundptilde}
  \left(1-c_{\gamma}\frac{\log \log N}{\log N}\right)  \left(\DP_{x}(G_{k,c_0}')- c_2N^{-c_1'}\right)\leq g_k(x) \leq \DP_{x}(G_{k,c_0}) + N^{-c_1}.
\end{equation}
\end{lemma}
\begin{proof}
  Let us first set the values of $c_0$ and $c_1$. 
  By \cite[Theorem 1.3]{Zaitsev98} and Markov’s property, we can,  for all $x\in \mathbb Z^2$,
 find a coupling $((S_n),(B_t),\DP_x$) and a $c_0$ large enough independent of $x$,  such that  $\DP_x(\Delta_N > c_0\log N) \leq N^{-c_1}$ with $c_1>0$ independent of $x$. We choose $c_0,c_1$ as such.
 
  We start with the upper bound in \eqref{eq:upperBoundptilde}.
With $F_k = \{\exists n\in \mathcal T_k|S_n = 0\}$,
\begin{align*}
g_k(x) &= \DP_{x} (F_k, \Delta_N \leq c_0\log N) + \DP_{x} (F_k, \Delta_N > c_0\log N)\\
& \leq \DP_{x}(G_{k,c_0}) + N^{-c_1}.
\end{align*}

We continue with the lower bound. In this case, unfortunately, knowing that the Brownian motion enters the ball of radius $c\log T$ does not imply necessarily that the random walk hits the origin. However, the random walk will be close enough to the origin so that its probability to hit the origin soon after is high. Denote by $T = \min\{n\in [t_1,t_2']\cap \mathbb N, |B_n|\leq 2c_0\log N\}$ and $T_0 = \inf\{n\geq 1|S_n = 0\}$. Let also $\theta_{k}$ stand for the shift in time of $k$ steps for the random walk. Since $t_2' = t_1+ l_k/2$, we have that on the event $T<\infty$,  $\{T_0\circ\theta_{T} < l_k/2\}\subset F_k$. Hence,
\begin{align*}
g_k(x) & \geq \DP_x\left(T<\infty, T_0\circ\theta_{T} < l_k/2,\Delta_T \leq c_0\log N\right)\\
& = \DE_x\left[\mathbf{1}_{T<\infty} \mathbf{1}_{\Delta_T \leq c_0\log N} \DP_{S_T}\left(T_0 < l_k/2\right)\right],
\end{align*}
where we have used Markov's property.  Since on $\{T<\infty, \Delta_T \leq c_0\log N\}$ we have $|S_T| \leq |B_T|+c_0\log N\leq 3c_0\log N$, we obtain that
\begin{align*}
g_k(x)\geq \inf_{|y|\leq 3c\log N} \DP_{y}\left(T_0 < l_k/2\right)  \DP_x\left(T<\infty, \Delta_T \leq c_0\log N \right).
\end{align*}
Observe that on $\{T<\infty\}$, one has $T\leq N$ and thus $\Delta_T\leq \Delta_N$, so that
\[
\DP_x\left(T<\infty, \Delta_T \leq c_0\log N \right) \geq \DP_x(T<\infty) - N^{-c_1}.
\]
The lower bound is thus proven if we show that
\begin{equation} \label{eq:continuityOfBM}
  \DP_x(T<\infty) \geq \DP_x(G_{k,c_0}')-N^{-c'},
\end{equation} for some $c'>0$, and that
\begin{equation} \label{eq:loglogN}
\inf_{|y|\leq 2c\log N} \DP_{y}\left(T_0 < l_k/2\right) \geq 1-c_{\gamma} \log \log N /\log N.
\end{equation}
For \eqref{eq:continuityOfBM}, we let $\Omega_N = \sup \{ |B_s-B_t|, |t-s|\leq 1, s,t\leq N\}$ and decompose
\begin{align*}
  \DP_x(G_{k,c_0}') & = \DP_x(G_{k,c_0}', \Omega_N< c_0\log N) + \DP_x(G_{k,c_0}', \Omega_N > c_0\log N)\\
  & \leq \DP_x(T<\infty) + \DP_x(\Omega_N > c_0\log N),
\end{align*}
so that \eqref{eq:continuityOfBM} follows from the fact that $\DP_x(\Omega_N>c_0\log N) \leq N^{-c'}$ for $c'>0$, see \cite[Theorem 3.2.4]{lawlerLimicRW}. We now prove \eqref{eq:loglogN}. 
We first use that since $l_k\geq N^{\gamma}$, 
we have $\DP_{y}(T_0 \geq l_k/2) \leq  \DP_{y}(T_0 \geq N^\gamma/2)$. Then, by \cite[Theorem 1]{RR66},
the last probability is smaller than $c_\gamma \log (2c_0 \log N)/\log N$ uniformly for $|y|\leq 2c_0\log N$.
\end{proof}

\begin{lemma} \label{lem:G_k}  Let $c_0$ be as in Lemma \ref{lem:KMT}. There exists $N_0=N_0(\tilde \Gamma)$ and $\Delta_{\Gamma,\ref{lem:G_k}} > 0$ such that 
$\Delta_{\Gamma,\ref{lem:G_k}} < 1$ for all $N>N_0$, $\limsup_{{\Gamma'}} \Delta_{\Gamma,\ref{lem:G_k}} = 0$ and for all $k\leq K$,
  \begin{align} \label{eq:controleBrownien1}
  \bar \alpha \left(1 - \Delta_{\Gamma,\ref{lem:G_k}}\right) &  \leq \inf_{x\in B(0,2\delta^{-1} L_k^{1/2})}\DP_x(G_{k,c_0}')\\
  & \leq \sup_{x\in B(0,2\delta^{-1} L_k^{1/2})} \DP_x(G_{k,c_0}) \leq  \bar \alpha (1+\Delta_{\Gamma,\ref{lem:G_k}}).\label{eq:controleBrownien2}
  \end{align}
\end{lemma}
\begin{proof} Let $x\in B(0,2\delta^{-1} L_k^{1/2})$. 
We begin with the second inequality (upper bound)  in \eqref{eq:controleBrownien2}. (The first inequality in \eqref{eq:controleBrownien2} is immediate). With $r_1 = c_0\frac{\log N}{\sqrt{t_1}}$, we have 
\begin{equation} \label{eq:BrowScaling}
  \DP_x(G_{k,c_0}) \leq \DP_0(G_{k,c_0}) = \DP_0\left(\inf_{s\in [1, t_2/t_1]} |B_s| \leq r_1\right),
\end{equation}
{where the first inequality follows from the fact that the modulus of the Brownian motion is a Bessel process and one can couple a Bessel process  $X_t^x$ started at $x$ to $B_t^0$
so that $B_t^x\geq B_t^0$ for all $t$, and the equality follows from Brownian scaling.}
In \cite{Spitzer58}, it is shown that
\[h_r(t) = \left(\log r^{-2}\right) \DP_0\left(\inf_{s\in [1,t]} |B_s| \leq r\right),\]
satisfies $h_r(t) \to \log t$ as $r\to 0$ for all fixed $t\geq 1$. 
Since $t\to h_r(t)$ is increasing and $t\to \log t$ is continuous, this convergence can be extended to a uniform convergence on each compact subset of $[1,\infty)$. By \eqref{eq:boundl_k} we have
\[t_2/t_1 = (\nu_1 l_{k-1}+l_k)/(\nu_1 l_{k-1}) \leq 1+e^{e\alpha},\] hence by the equality in \eqref{eq:BrowScaling},
\begin{equation} \label{eq:asym_P0}
\left| \left(\log r_1^{-2}\right) \DP_0(G_{k,c_0}) - \log t_2/t_1 \right| \leq \sup_{t\in [1,1+e^{e\alpha}]}|h_{r_1}(t)-\log t| =: \varepsilon_{N},
\end{equation}
where $\varepsilon_{N} = \varepsilon_N(\alpha,\gamma,\nu_1) \to 0$ as $N\to\infty$ since $r_1$ vanishes as $N\to\infty$. Moreover, by \eqref{eq:ratioOnl_k}, there exists $\varepsilon_N'=\varepsilon_N'(\alpha,\gamma) \to 0$ as $N\to\infty$ such that $\varepsilon_N' >0$ and
\begin{align*}
\log t_2/t_1&  = \log(l_k/(\nu_1 l_{k-1})) + \log(1+\nu_1l_{k-1}/l_k) \leq (e^{\bar \alpha} - 1) \log l_{k-1} + \log 2 + \varepsilon_N',
\end{align*}
where 
we have used that $\nu_1 l_{k-1}/l_k \leq 1$ (Remark \ref{rem:2lk}).
Hence, by \eqref{eq:asym_P0}, we find that
\begin{align*}
\DP_0(G_{k,c_0}) &\leq \frac{(e^{\bar \alpha}-1)\log l_{k-1} +\log 2 +\varepsilon_N+\varepsilon_N'}{\log l_{k-1} + \log (\nu_1/(c_0^2(\log N)^2))}\leq (e^{\bar \alpha}-1)\frac{\left(1+ \frac{\log 2 +\varepsilon_N+\varepsilon_N'}{(e^{\bar \alpha} - 1)\log l_{k-1}}\right)}{1 - 2 \frac{\log (c_0\log N)}{\log l_{k-1}}}.
\end{align*}
Since $\log l_{k-1} \geq \gamma\log N$, the numerator is smaller than  $1 + \frac{\log 2 + \varepsilon_N+\varepsilon_N'}{\gamma \alpha}$ and
for $\alpha$, $\gamma$ and $\nu_1$ fixed, the denominator writes as $(1+o_N(1))$.
This gives \eqref{eq:controleBrownien2}.

We turn to \eqref{eq:controleBrownien1}. For $r_1 = c_0 \log N / \sqrt {t_1}$, by {Brownian scaling and Markov's property,}
\begin{align}
\DP_x(G_{k,c_0}') & = \DP_{x/\sqrt {t_1}}\left(\inf_{s\in [1, t_2'/t_1]} |B_s| \leq r_1\right)\nonumber\\
& = \int_{\mathbb R^2} \bar p_1(z-x/\sqrt t_1) \DP_z\left(\inf_{s\in [0,t_2'/t_1 - 1]} |B_s| \leq r_1\right)\dd z,\label{eq-integral}
\end{align}
where $\bar p_{t}(x) = \frac{1}{\pi t} e^{-|x|^2/t}$.
Then, we have:
\[\bar p_1(z-w) \geq \bar p_1(z) e^{-2|w||z| -|w|^2} \geq \bar p_1(z) e^{-|w|(|z|^2+1) -|w|^2},
\]
 by the Cauchy-Schwarz inequality and the bound $|z| \leq (|z|^2 + 1)/2$ valid for all $z$. This implies that for all $x\in B(0,2\delta^{-1} L_k^{1/2})$, with $\xi^2 = 4\delta^{-2}  L_k / t_1$,
\begin{align*}
\bar p_1(z-x/\sqrt t_1) &\geq \bar p_1(z) e^{- \xi |z|^2} e^{-\xi^2-\xi}  
= (1+\xi)^{-1} \bar p_{\frac{1}{1+\xi}}(z)  e^{-\xi^2-\xi}  .
\end{align*}
Plugging the last inequality in the integral {in \eqref{eq-integral},} we obtain that for $u_1 = (1+\xi)^{-1}$ and $u_2 = t_2'/t_1 - \xi/(1+\xi)$,
\begin{equation} \label{eq:lowerbound_gk'}
\DP_x(G_{k,c_0}')
 \geq  \frac{e^{-\xi^2-\xi}}{1+\xi} \DP_0\left(\inf_{s\in [u_1,u_2]} |B_s| \leq r_1\right).
\end{equation} 
By \eqref{eq:boundl_k}-(ii), we see that $\xi^2\leq  16 \nu_1^{-1} \delta^{-2}\nu_2$, in particular $\xi^2\leq 1/2$ by \eqref{eq:assumptionOnParameters}-(iii). Similarly to \eqref{eq:asym_P0}, we obtain that
\[
\left | \log  (r_1/u_1) \DP_0\left(\inf_{s\in [u_1,u_2]} |B_s| \leq r_1\right) - \log (u_2/u_1) \right| \leq \varepsilon_{N},
\]
where $\varepsilon_{N} = \varepsilon_N(\alpha,\gamma,\nu_1,\nu_2,\delta) \to 0$ as $N\to\infty$.
As $u_2/u_1 = (1+\xi) t_2'/t_1- \xi$,
\begin{align*}
\log u_2/u_1 & = \log t_2'/t_1 + \log\left(1+ \xi - \xi t_1/t_2'\right)\\
& \geq \log(l_k/(2\nu_1 l_{k-1})) + \log(1-\xi)\\
& \geq (e^{\bar \alpha} - 1) \log l_{k-1} -\log (2\nu_1)-\log 2-\varepsilon_N',
\end{align*}
where we have used \eqref{eq:ratioOnl_k} and that $\xi^2\leq 1/2$.
Thus, as $(e^{\bar \alpha} - 1)\geq \bar \alpha$,
\begin{align*}
& \DP_0\left(\inf_{s\in [u_1,u_2]} |B_s|  \leq r_1\right) \geq \frac{\log (u_2/u_1)-\varepsilon_N}{\log r_1 - \log u_1}\\
& \geq \frac{\bar \alpha  \log l_{k-1} -\log (4\nu_1)-\varepsilon_N-\varepsilon_N'}{\log l_{k-1} + \log (2\nu_1/(c_0^2(\log N)^2))} \geq \bar \alpha \frac{1-\frac{\log (4\nu_1)+\varepsilon_N +\varepsilon_N'}{\bar \alpha \log l_{k-1}}}{1+ \frac{\log (2 \nu_1)}{\log l_{k-1}}}.
\end{align*}
For fixed $\nu_1, \alpha,\gamma$, the denominator is $1+o_N(1)$ as $N\to\infty$. As $\bar \alpha \log l_{k-1} \geq \gamma \alpha$, we obtain from \eqref{eq:lowerbound_gk'} that
\[
  \DP_x(G_{k,c_0}') \geq \bar \alpha e^{-\xi^2-\xi}\frac{1}{1+\xi}  \frac{1-\frac{\log (4\nu_1)+\varepsilon_N+\varepsilon_N'}{\gamma \alpha}}{1+ o_N(1)}.
\]
This implies \eqref{eq:controleBrownien1} since $\xi^2\leq  16 \nu_1^{-1} \delta^{-2}\nu_2$.
 The condition $\Delta_{\Gamma,\ref{lem:G_k}} < 1$ for $N$ large enough is ensured by \eqref{eq:assumptionOnParameters}-(ii).
\end{proof}
We are now ready to complete the proof of Proposition \ref{prop:pxy}.
\begin{proof}[Proof of Proposition \ref{prop:pxy}]
  The result follows from combining Lemma \ref{lem:removeConditioning}, Lemma \ref{lem:remove_barrier},  Lemma \ref{lem:KMT} (with \eqref{eq:identityptog}) and Lemma \ref{lem:G_k}.
  \end{proof}

\subsection{Three-particle intersection probability} In this section, we derive an upper bound on the probability $p_{(i,j);(i',j')} = \DP_{\mathbf x}^{o_k,\mathbf y} (\tau_k^{(i,j)}<\infty, \tau_k^{(i',j')}<\infty)$ when $\{i,j\}\cap \{i',j'\}\neq \emptyset$, $\mathbf x\in B_{{q_0},k}$ and $\mathbf y\in B_{{q_0},k+1}$. By symmetry, it is enough to control $p^{(3)}_{\mathbf w,\mathbf z} := \DP_{\mathbf w}^{o_k,\mathbf z}(\tau_k^{(1,2)} < \infty, \tau_k^{(1,3)}<\infty)$ for $\mathbf w\in B_{3,k}$ and $\mathbf z\in B_{3,k+1}$.

The result is the following.
\begin{proposition} \label{prop:3intersections} There exists $C > 0$ and $\Delta_{\Gamma,\ref{prop:3intersections}} >0$ satisfying 
$\limsup_{\Gamma'}\Delta_{\Gamma,\ref{prop:3intersections}}=0$
such that for all $k\leq K$, $\mathbf x\in B_{3,k}, \mathbf y\in B_{3,k+1}$,
\begin{align*}
p^{(3)}_{\mathbf x,\mathbf y}\leq C \bar \alpha \frac{\log \log N}{\gamma \log N} (1+\Delta_{\Gamma,\ref{prop:3intersections}}).
\end{align*}
\end{proposition}
Before turning to the proof of Proposition \ref{prop:3intersections}, we state a few lemmas.
As in the previous section, we first observe that we can forget about the conditioning on the endpoints. Letting $p^{(3)}_{\mathbf x} = \DP_{\mathbf x}(\tau_k^{(1,2)}<\infty, \tau_k^{(1,3)}<\infty)$ for $\mathbf x \in B_{3,k}$, we have:
\begin{lemma} \label{lem:removeConditioning3} There exists $\Delta_{\Gamma,\ref{lem:removeConditioning3}} >0$ satisfying
$\limsup_{\Gamma'} \Delta_{\Gamma,\ref{lem:removeConditioning3}} = 0$ such that
for all $k\leq K$,  $\mathbf x\in B_{3,k}$ and $\mathbf y\in B_{3,k+1}$, we have $p^{(3)}_{\mathbf x,\mathbf y} \leq p^{(3)}_{\mathbf x}(1+ \Delta_{\Gamma,\ref{lem:removeConditioning3}})$.
\end{lemma} 
We omit the proof which is very similar to the one of Lemma \ref{lem:removeConditioning}. 

Next, we let $T_0 = \inf\{n\geq 0 : S_n = 0\}$. The following holds.
\begin{lemma} \label{lem:h_k} Let $h_k(x) = \DP_x(T_0 \leq \ell_k)$. There exists $c=c(\gamma,\alpha,\nu_1)>0$ such that
  \[
  \sup_{n\in \mathcal T_k} \sup_{x\in \mathbb Z^2} \DE_0[h_k(S_n - x)] \leq c (\log N)^{-4} + C \frac{\log \log N}{\gamma \log N}.
  \]
  \end{lemma}
  \begin{proof} Let $n\in \mathcal T_k$ and $x\in \mathbb Z^2$. By \cite[Théorème 3.6]{LeGall86}, we have $(\log l_k)h_k(z) \leq C \left(\log ({l_k} |z|^{-2})\right)_+ + C \mathbf{1}_{|z|^2\geq l_k}$.
  Hence we decompose 
  \begin{align*}
  &\DE_0[h_k(S_n - x)]
  \\& = \sum_{|z-x|\leq (\log N)^{-2} \sqrt {l_k}} p_n(z) h_k(z-x) + \sum_{|z-x| > (\log N)^{-2} \sqrt {l_k}} p_n(z) h_k(z-x)\\
  & \leq  \frac{C}{\nu_1 l_{k-1}} (\log N)^{-4} l_k + \frac{C\log \log N}{\log l_k} +\frac{C}{\log l_k},
  \end{align*}
  where in the first sum we have bounded $h_k$ by 1 and used that $p_n(z)\leq \frac{C}{n}$ (Corollary \ref{cor:p_nbound}) with $n\geq \nu_1 l_{k-1}$. The proof is completed using \eqref{eq:boundl_k}-(i) and $l_k\geq N^\gamma$.
\end{proof}

\subsection{Proof of Proposition \ref{prop:3intersections}}
Let $\tilde \tau_k^{(i,j)} = \inf\{n\in \mathcal T_k : S_n^i=S_n^j\}$ and $\tilde p^{(3)}_{\mathbf x} = \DP_{\mathbf x} (\tilde \tau_k^{(1,2)} < \infty, \tilde \tau_k^{(1,3)}<\infty)$. It then trivially holds that 
$p^{(3)}_{\mathbf x} \leq \tilde p^{(3)}_{\mathbf x}$. Furthermore, by symmetry, 
\begin{align*}
\tilde p^{(3)}_{\mathbf x} & \leq 2 \DP_{\mathbf x}^{\otimes 3}\left(\tilde \tau_k^{(1,2)} \leq \tilde \tau_k^{(1,3)}<\infty \right) \leq 2 \DP_{\mathbf x}^{\otimes 3}\left(\tilde \tau_k^{(1,2)} <\infty, T^{(1,3)} \circ \theta_{\tilde \tau_k^{(1,2)}} \leq \ell_k\right),
\end{align*}
where $T^{(1,3)} = \inf\{n\geq 0|S_n^1 = S_n^3\}$ and $\theta_k$ denotes shift in time of $k$ steps. Let $T_0 = \inf\{n\geq 0 : S_n = 0\}$ and $h_k(x) = \DP_x(T_0 \leq \ell_k)$. By Markov's property,
\begin{align*}
&\tilde p^{(3)}_{\mathbf x} \leq   \DE_{\mathbf x}^{\otimes 3} \left[\mathbf{1}_{\tilde \tau_k^{(1,2)} < \infty} \ h_k\left(S^1_{\tilde \tau_k^{(1,2)}} - S^3_{\tilde \tau_k^{(1,2)}}\right) \right]\\
& = \DE_{\mathbf x}^{\otimes 3}\left[\mathbf{1}_{\tilde \tau_k^{(1,2)} < \infty} \ \DE_{\mathbf x}^{\otimes 3} \left[h_k\left(S^1_{\tilde \tau_k^{(1,2)}} - S^3_{\tilde \tau_k^{(1,2)}}\right) \middle | S^1,S^2 \right] \right].
\end{align*}
Then, combine Lemma \ref{lem:h_k} with the identity $\DE_{\mathbf x}(\tilde \tau_k^{(1,2)} < \infty) = g_k(x_2-x_1)$ and the upper bounds in Lemma \ref{lem:KMT} and Lemma \ref{lem:G_k} to obtain that
 \[
  \tilde p^{(3)}_{\mathbf x}\leq \left(\bar \alpha  (1+\Delta_{\Gamma,\ref{lem:G_k}}) + N^{-c_1}\right) \left( C\frac{\log \log N}{\gamma \log N}(1+o_N(1))\right),
  \]
with $\limsup_{\Gamma'} \Delta_{\Gamma,\ref{lem:G_k}}=0$.
We conclude with Lemma \ref{lem:removeConditioning3}.
\qed

\appendix
\section{Local central limit theorem}
\label{AppendixLCLT}
Let $p_t(x)$ be the probability transition function of the simple random walk on $\mathbb Z^d$ and $\bar p_{t}(x) = \frac{1}{\pi t} e^{-\frac{d|x|^2}{2t}}$. We say that $x\sim_{t}y$ when $x$ and $y$ have the same parity, that is $p_t(x-y) > 0$. The following theorem can be obtained from Theorem 2.3.11 in \cite{lawlerLimicRW}. (See also the proof of \cite[Theorem 2.1.3]{lawlerLimicRW} and the paragraph above the statement of that theorem).
\begin{theorem}[Local central limit theorem] \label{th:LCLT}
There exists $\rho > 0$ such that for all $t\geq 0$ and all $x\in \mathbb Z^d$ with $|x|< \rho t$ and $x\sim_t 0$,
\begin{equation}
  p_t(x) = 2 \bar p_t(x) \exp \left\{O\left(\frac{1}{t} + \frac{|x|^4}{t^3}\right) \right\},
\end{equation}
where $O(g)$ satisfies  $|O(g)|\leq C |g|$ for some universal constant $C>0$.
\end{theorem}
Since $p_{2n}(x)$ is maximal at $x=0$ (this is a direct consequence of the Cauchy-Schwarz inequality), the theorem implies the next general bound:
\begin{corollary} \label{cor:p_nbound} There exists $C>0$ such that for all $n\geq 1$,
$\sup_{x\in\mathbb Z} p_n(x) \leq \frac{C}{n}$.
\end{corollary}

\section{The case $1 \ll q^2 \leq \log \log N$} \label{appendixB}
In the regime $1 \ll q^2 \leq \log \log N$, the error in Proposition \ref{prop:D_N} becomes too large. 
The reason is that we ask for many particles to meet in a ball at 
each time $L_k$, and there are around $\log N$  such times. This event
has a cost which is too big compared to the value of the moment $\IE[W_N^q]$ when $q\leq \log \log N$.  To fix this issue, we can divide $[0,N]$ into less intervals $[L_k,L_{k+1})$ by defining $\bar \alpha = \frac{\alpha}{\binom q 2}$ instead of $\bar \alpha = \frac{\alpha}{\log N}$. With this change, the error term in Proposition \ref{prop:D_N} can be neglected. However, when choosing $\bar \alpha = \frac \alpha {\binom q 2}$, the quantity $t_2/t_1$ diverges in the proof of Lemma \ref{lem:G_k}, so that we cannot resctrict ourselves to a compact set in order to extend the pointwise convergence of {\cite{Spitzer58}} to a uniform one in the argument for \eqref{eq:asym_P0}. Hence, we need to extend the
{main result in \cite{Spitzer58}} to allow for a uniform control on time and space. There exist uniform results by Ridler and Rowe {\cite{RR66},}
both for the random walks and the Brownian motion, but they are given for the quantity $\DP_x(T>n)$ (where $T$ is the first return time to $0$) instead of $\DP_x(T<n)$ that we need, and unfortunately the error term given is not enough to go from one quantity to the other in our case.

The following lemma deals with this problem. It is obtained by adapting the arguments of Spitzer \cite{Spitzer58} and Ridler-Rowe \cite{RR66}.  Consider the Bessel process $R_t = |B_t|$ and define $T_a=\inf \{t\geq 0 : R_t = a\}$.
\begin{lemma}  \label{lem:Ridler-Rowe}
For all $c>0$, it holds that
  \[\DP_r(T_a \leq t) = \frac{\log(t/r^2)}{\log(t/a^2)} (1+o(1)),\] where the error term $o(1)$ vanishes as $r^2/t\to 0$ uniformly for $a<r$.
  \end{lemma}
\begin{proof}
 The goal is to deduce bounds on $\DP_r(T_a\leq t)$ from its Laplace transform
\[A(a,r,\lambda)= \int_{0}^\infty e^{-\lambda t} \DP_r(T_a\leq t) \mathrm d t, \quad a<r, \lambda>0.\]
We follow closely the approach used in \cite[Main Proof and Proof of Theorem 2]{RR66} which is based on a Karamata method of obtaining Tauberian theorems.
The starting point is the next formula, proved in \cite[(1.4)]{Spitzer58},
\begin{equation}  \label{eq:SpitzerLaplaceFormula}
  A(a,r,\lambda) = \frac{K_0\left(r \sqrt{2\lambda}\right)}{\lambda K_0\left(a\sqrt{2\lambda}\right)}, \quad \text{with} \quad K_0(u) = -\log u + C + O(u ) \quad \text{as } u\to 0.
\end{equation}
In particular, it holds that
\begin{equation} \label{eq:SpitzerLaplace}
A(a,r,\lambda) = \frac{1}{\lambda} \frac{\log (r^2 \lambda)}{\log(a^2 \lambda)} (1+o(1)),
\end{equation}
where $o(1)$ vanishes as $r^2 \lambda \to 0$ uniformly for $a<r$.
Then, the idea is to relate $\DP_r(T\leq t)$ to its Laplace transform via the formula
\begin{equation} \label{eq:defB}
B(a,r,\lambda^{-1}):= \int_0^{\lambda^{-1}} \DP_{r}(T_a\leq t)\dd t = \int_{0}^\infty e^{-\lambda t} g(e^{-\lambda t}) \DP_{r}(T_a\leq t)  \dd t,
 \end{equation}
 where $g(u) = u^{-1}$ when $e^{-1} \leq u \leq 1$ and $0$ otherwise. 
We will first obtain bounds on $B(a,r,t)$ and deduce a bound on its $t$-derivative $\DP_r(T_a\leq t)$ in a second step. Let $\varepsilon\in (0,1)$ and
\[
Q(u) = \sum_{n=0}^{m} a_n u^n \quad \text{and} \quad R(u) = \sum_{n=0}^{l} b_n u^n,
\]
be two polynomials satisfying 
\begin{equation} \label{eq:EpsilonPolynomials}
  Q \leq g \leq R \text{ on } [0,1]  \text{ and } \Vert Q-R\Vert_{1,[0,1]} = \int_{0}^\infty e^{-t}\left(R(e^{-t})-Q(e^{-t}) \right) \dd t < \varepsilon.
\end{equation}
By \eqref{eq:defB}, we have
\begin{align*}
B(a,r,\lambda^{-1}) & \geq \int_{0}^\infty e^{-\lambda t} Q(e^{-\lambda t}) \DP_{r}(T_a\leq t) \dd t\\
& = \sum_{n=0}^m a_n \int_{0}^\infty e^{-(n+1)\lambda t} \DP_{r}(T_a\leq t) \dd t\\
& = \sum_{n=0}^m a_n A(a,r,(n+1)\lambda).
\end{align*}
Now by \eqref{eq:SpitzerLaplace}, we can find $\delta_\varepsilon>0$ such that whenever $r^2 \lambda \leq \delta_\varepsilon$, we have
\begin{equation} \forall n\leq m, \quad 
\lambda^{-1}(1-\varepsilon) \leq \frac{\log (a^2 \lambda) }{\log (r^2 \lambda)} A(a,r,(n+1)\lambda) \leq \lambda^{-1} (1-\varepsilon),
\end{equation}
uniformly for all $a<r$. This implies that
\begin{equation} \label{eq:lowerBoundB}
  \frac{\log (a^2 \lambda) }{\log (r^2 \lambda)}  B(a,r,\lambda^{-1}) \geq (1-\varepsilon) \lambda^{-1} \sum_{n=0}^m \frac{a_n}{n+1} \geq (1-\varepsilon)^2 \lambda^{-1},
\end{equation}
where in the second inequality we have used \eqref{eq:EpsilonPolynomials} and $\int_0^\infty e^{-t} g(e^{-t}) \dd t = 1$ to obtain
\[
  \sum_{n=0}^m \frac{a_n}{n+1} = \int_{0}^\infty e^{-t} \sum_{n=0}^m a_n e^{-nt} \dd t = \int_{0}^\infty e^{-t} Q(e^{-t}) \dd t \geq 1-\varepsilon.
\]
{A similar computation leads to an upper bound in \eqref{eq:lowerBoundB},
 with $1-\varepsilon$ replaced by $1+\varepsilon$.}
 Hence, setting $\lambda^{-1} = t$, we obtain that for all $a < r$ and $r^2/t \leq \delta_\varepsilon$,
\begin{equation} \label{eq:FullboundB}
 t(1-\varepsilon) \leq  \frac{\log (t/a^2) }{\log (t/r^2)}  B(a,r,t) \leq t(1+\varepsilon).
\end{equation}

Now, since $t\to \DP_r(T_a\leq t)$ is non-decreasing, we have for all $\delta>0$,
\[
  \frac{B(a,r,t)}{t}\leq \DP_r(T_a \leq t) \leq \frac{B(a,r,t^{1+\delta})-B(a,r,t)}{t^{1+\delta}-t}.
\]
By \eqref{eq:FullboundB}, this leads to the following bound valid for $r^2/t\leq \delta_\varepsilon$ and $a<r$, 
\[
  1-\varepsilon \leq \frac{\log(t/a^2)}{\log(t/r^2)}\DP_r(T_a \leq t)\leq
 1+\varepsilon \frac{1+t^{-\delta}}{1-t^{-\delta}}
  +\delta(1+\varepsilon) \frac{(\log t)/\log(t/r^2)}{1-t^{-\delta}}.
\]
We thus choose $\delta = \varepsilon \frac{\log(t/r^2)}{\log t}$ and observe that $t^{-\delta} = e^{-\varepsilon \log(t/r^2)}$ so that
\[
  1-\varepsilon \leq \frac{\log(t/a^2)}{\log(t/r^2)}\DP_r(T_a \leq t)\leq
  1+3\varepsilon,
\]
when $r^2/t \leq \delta_\varepsilon$ up to choosing $\delta_\varepsilon$ smaller.
\end{proof}
Building up on Lemma \ref{lem:Ridler-Rowe}, we can deduce the following.
\begin{lemma} \label{lem:SpitzerExtended}
  There exists a constant $C_0>0$ such that
\[
\DP_0\left(\inf_{t\in [t_1,t_2]} R_t \leq a \right) = \frac{\log(t_2/t_1)}{\log(t_2/a^2)}(1+o(1))  +  h_0(t_1,t_2,a),
\]
where the error term $o(1)$ vanishes as $t_2/t_1 \to \infty$ uniformly for $a^2<t_1$ and $|h_0(t_1,t_2,a)| \leq a^2/t_1$. 
\end{lemma}
\begin{proof}
  Let $\varepsilon$ and $\delta$ in $(0,1)$. {(We choose below
  $\delta$ small as function of $\varepsilon$.)}
By Markov's property,
\begin{align}
  \DP_0\left(\inf_{t\in [t_1,t_2]} R_t \leq a \right) & = \frac{1}{t_1}\int_0^a re^{-r^2/(2t_1)} \dd r \label{eq:1stlineIntegral}\\
  & + \frac{1}{t_1}\int_a^{\sqrt{\delta^{-1} t_1}} re^{-r^2/(2t_1)}\DP_{r}\left( T_a \leq t_2-t_1\right) \dd r \label{eq:2ndlineIntegral}\\
  &\quad  + \frac{1}{t_1}\int_{\sqrt{\delta^{-1} t_1}}^\infty re^{-r^2/(2t_1)}\DP_{r}\left( T_a \leq t_2-t_1\right)  \dd r, \label{eq:3rdlineIntegral}
\end{align}
where $\sqrt{\delta^{-1}t_1} > a$ since $a^2<t_1$ by assumption. 
First observe that the integral on the right-hand side of \eqref{eq:1stlineIntegral} is smaller than $a^2/t_1$. Next,
let $\delta$ be small enough so that by Lemma \ref{lem:Ridler-Rowe}, we have for all $r^2/(t_2-t_1)\leq \delta$ that 
\begin{equation} \label{eq:ProbaTat2}
  \DP_r(T_a \leq t_2-t_1) = \frac{\log((t_2-t_1)/r^2)}{\log((t_2-t_1)/a^2)}(1+e_{t_1,t_2,a,r}),
\end{equation}
with $|e_{t_1,t_2,a,r}|\leq \varepsilon$ uniformly for $a <r$.
We now assume that $t_1\leq M^{-1} t_2$ with $M>M_\delta$ large enough so that $t_1/(t_2-t_1) < \delta^2$ {(we also require $M$ to be large enough so that
\eqref{eq-star} below hold).} It then holds that $r^2/(t_2-t_1) \leq \delta$ in the integral \eqref{eq:2ndlineIntegral}, which is thus equal to
\begin{equation*} 
(1+e_{t_1,t_2,a,\delta})\frac{1}{t_1} \int_{a}^{\sqrt{\delta^{-1} t_1}} re^{-r^2/(2t_1)} \frac{\log((t_2-t_1)/r^2)}{\log((t_2-t_1)/a^2)} \dd r,
\end{equation*}
where $|e_{t_1,t_2,a,\delta}|\leq \varepsilon$. Write the last integral as $I_1-I_2$, where
\begin{align*}
  I_1 = \frac{1}{t_1} \int_{a}^{\infty} re^{-r^2/(2t_1)} \frac{\log((t_2-t_1)/r^2)}{\log((t_2-t_1)/a^2)} \dd r,
\end{align*}
and $I_2$ is the same integral between $\sqrt{\delta^{-1}t}$ and $+\infty$.
By integration by part,
\begin{align*}
  \log \left(\frac{t_2-t_1}{a^2}\right) I_1 &= e^{-\frac{a^2}{2t_1}} \log\frac{t_2-t_1}{a^2} -  \int_{\frac{a^2}{2t_1}}^\infty \frac{e^{-v}}{v}\dd v\\
 & =  e^{-\frac{a^2}{2t_1}}\log\frac{t_2-t_1}{a^2}+\log \frac{a^2}{2t_1}+\gamma +O\left(\frac{a^2}{t_1}\right)\\
 &= \log \frac{t_2-t_1}{t_1} + C_0 + O\left(\frac{a^2}{t_1}\log\frac{t_2-t_1}{a^2}\right),
\end{align*}
where $\gamma$ is the Euler constant and $C_0 = \gamma-\log 2$. Therefore,
\[I_1 = \frac{\log((t_2-t_1)/ t_1) + C_0}{\log((t_2-t_1)/a^2)} + O\left(\frac{a^2}{t_1}\right).
\]
{Note that the implicit constant in the error term $O(a^2/t_1)$ does not depend on $\delta$.}

Next, there exists $c>0$ such that
\begin{equation*}
   I_2 = \frac{1}{t_1} \int_{\sqrt{\delta^{-1} t_1}}^\infty re^{-r^2/(2t_1)} \frac{\log((t_2-t_1)/r^2)}{\log((t_2-t_1)/a^2)}\dd r \leq e^{-c \delta^{-1}} \frac{\log((t_2-t_1)/ t_1)}{\log((t_2-t_1)/a^2)},
\end{equation*}
The integral in \eqref{eq:3rdlineIntegral} can be bounded in the same way using that for all $r>\sqrt{\delta^{-1}t_1}$ we have $\DP_r(T_a \leq t_2-t_1)\leq \DP_{t_1^{1/2}} (T_a \leq t_2-t_1)$ and applying \eqref{eq:ProbaTat2}. Finally, note that for $M$ large enough,
\begin{equation}
  \label{eq-star}
  \frac{\log((t_2-t_1)/ t_1)+C_0}{\log((t_2-t_1)/a^2)} = \frac{\log(t_2/t_1)}{\log(t_2/a^2)}(1+e_{t_1,t_2,a}),
  \end{equation}
where $|e_{t_1,t_2,a}|\leq \varepsilon$ uniformly for $a^2< t_1$. Putting everything together, we find that
\[
  \DP_0\left(\inf_{t\in [t_1,t_2]} R_t \leq a \right) = \frac{\log(t_2/t_1)}{\log(t_2/a^2)}\left(1+e_{t_1,t_2,a} - {e'}_{t_1,\delta,a}\right) + O\left(\frac{a^2}{t_1}\right),
\]
where $|{e'}_{t_1,\delta,a}| \leq e^{-c\delta^{-1}}$. This concludes the proof since 
$\varepsilon$, {and then $\delta$,}
can be taken arbitrary small, {as long as $t_2/t_1>M_\delta$.}
\end{proof}
\paragraph{\textbf{Adapting the proof of Lemma \ref{lem:G_k}.}} With the last lemma, we can adapt the proof of Lemma \ref{lem:G_k} to the case $\bar \alpha = \alpha / \binom q 2$ with $1\ll q^2 \leq \log \log N$ as follows. For simplicity we consider  $\DP_x(G_{k,c_0})$ for $x= 0$, the reduction from $x$ in the ball to $x=0$ can be done as in the proof of the Lemma \ref{lem:G_k}. Recall that $t_1 = \nu_1 l_{k-1}$ and $t_2 = t_1 + l_k$. We have 
\[
  \DP_0(G_{k,c_0}) = \DP_0\left(\inf_{t \in [t_1,t_2]} |B_t| \leq c_0 \log N\right) \text{ with }   t_2/t_1 \geq N^{\gamma \bar \alpha} \geq e^{\gamma \alpha \log N / \log \log N}
\]
and $t_1 \geq N^\gamma$, so that Lemma \ref{lem:SpitzerExtended} applies. Moreover,
\begin{equation} \label{eq:estimeesRatiosLogs}
  \log (l_k/l_{k-1}) \sim \bar \alpha  e^{\bar \alpha (k-1)}\gamma   \log N \text{ and } \log l_k \sim e^{\bar \alpha k} \gamma \log N, \quad N\to\infty.
\end{equation}
We obtain that
\begin{align*} \DP_0(G_{k,c_0}) &= \frac{\log(t_2/t_1)}{\log(t_2/a^2)} (1+o(1)) \,  +  O(a^2/t_1) \\
  & = \frac{\log (l_k/l_{k-1})}{\log l_k}(1+o(1)) + O\left((\log N)^2/N^\gamma\right)
  = \bar \alpha + o(\bar \alpha),
\end{align*}
by \eqref{eq:estimeesRatiosLogs}. This recovers Lemma \ref{lem:G_k}.

\bibliographystyle{plain}
{\footnotesize \bibliography{polymeres-bib}
}
\end{document}